\documentclass{amsart}
\usepackage{amssymb}
\usepackage{amsfonts}
\usepackage{amsthm}
\usepackage{mathtools}
\usepackage{todonotes}
\usepackage{mathscinet}%% for cprime and compilation of russian names

\calclayout

\def\Xint#1{\mathchoice
	{\XXint\displaystyle\textstyle{#1}}%
	{\XXint\textstyle\scriptstyle{#1}}%
	{\XXint\scriptstyle\scriptscriptstyle{#1}}%
	{\XXint\scriptscriptstyle\scriptscriptstyle{#1}}%
	\!\int}

\def\XXint#1#2#3{{\setbox0=\hbox{$#1{#2#3}{\int}$ }
		\vcenter{\hbox{$#2#3$ }}\kern-.6\wd0}}

\usepackage{esint}

\numberwithin{equation}{section}
\theoremstyle{plain}
\newtheorem{theorem}[equation]{Theorem}
\newtheorem{proposition}[equation]{Proposition}
\newtheorem{lemma}[equation]{Lemma}
\newtheorem{corollary}[equation]{Corollary}

\newtheorem{example}[equation]{Example}

\theoremstyle{remark}
\newtheorem{remark}[equation]{Remark}

\theoremstyle{definition}

\newtheorem*{question*}{Question}

\usepackage{datetime} 

\usepackage{graphicx}

\usepackage{url}

\def\vint_#1{\mathchoice%
        {\mathop{\kern 0.2em\vrule width 0.6em height 0.69678ex depth -0.58065ex
                \kern -0.8em \intop}\nolimits_{\kern -0.4em#1}}%
        {\mathop{\kern 0.1em\vrule width 0.5em height 0.69678ex depth -0.60387ex
                \kern -0.6em \intop}\nolimits_{#1}}%
        {\mathop{\kern 0.1em\vrule width 0.5em height 0.69678ex
            depth -0.60387ex
                \kern -0.6em \intop}\nolimits_{#1}}%
        {\mathop{\kern 0.1em\vrule width 0.5em height 0.69678ex depth -0.60387ex
                \kern -0.6em \intop}\nolimits_{#1}}}
\def\vintslides_#1{\mathchoice%
        {\mathop{\kern 0.1em\vrule width 0.5em height 0.697ex depth -0.581ex
                \kern -0.6em \intop}\nolimits_{\kern -0.4em#1}}%
        {\mathop{\kern 0.1em\vrule width 0.3em height 0.697ex depth -0.604ex
                \kern -0.4em \intop}\nolimits_{#1}}%
        {\mathop{\kern 0.1em\vrule width 0.3em height 0.697ex depth -0.604ex
                \kern -0.4em \intop}\nolimits_{#1}}%
        {\mathop{\kern 0.1em\vrule width 0.3em height 0.697ex depth -0.604ex
                \kern -0.4em \intop}\nolimits_{#1}}}

\newcounter{prob}
\newcommand{\Z}{\ensuremath{\mathbb{Z}}}

\newcommand{\N}{\ensuremath{\mathbb{N}}}
\newcommand{\R}{\ensuremath{\mathbb{R}}}

\newcommand{\diam}{\ensuremath{\ \mathrm{Diam\ }}}

\newcommand{\defeq}{\mathrel{\mathop:}=}

\newcommand{\A}{\ensuremath{\mathcal{A}}}
\newcommand{\Rp}{\ensuremath{\mathcal{R}_p}}
\newcommand{\Ri}{\ensuremath{\mathcal{R}_1}}

\newcommand{\cH}{\ensuremath{\mathcal{H}}}

\newcommand{\loc}{\ensuremath{\mathrm{loc}}}

\newcommand{\Wdp}{\ensuremath{\dot{W}^{1,p}}}

\def\Xint#1{\mathchoice
{\XXint\displaystyle\textstyle{#1}}%
{\XXint\textstyle\scriptstyle{#1}}%
{\XXint\scriptstyle\scriptscriptstyle{#1}}%
{\XXint\scriptscriptstyle\scriptscriptstyle{#1}}%
\!\int}
\def\XXint#1#2#3{{\setbox0=\hbox{$#1{#2#3}{\int}$ }
\vcenter{\hbox{$#2#3$ }}\kern-.58\wd0}}
\def\avint{\Xint-}

\usepackage[pagebackref,colorlinks,citecolor=blue,linkcolor=blue]{hyperref}

\title{On Limits at Infinity of Weighted Sobolev Functions}

\author{Sylvester Eriksson-Bique}
\address{Research Unit of Mathematical Sciences,
P.O.Box 8000,
FI-90014 Oulu, Finland}
\email{\tt sylvester.eriksson-bique@oulu.fi}

\author{Pekka Koskela}
\address{Department of Mathematics and Statistics \\
P.O. Box 35 \\
FI-40014 University of Jyväskylä}
\email{pekka.j.koskela@jyu.fi}

\author{Khanh Nguyen}
\address{Department of Mathematics and Statistics \\
P.O. Box 35 \\
FI-40014 University of Jyväskylä}
\email{khanh.n.nguyen@jyu.fi}


\begin{document}

\maketitle
\begin{center}
{ \textit{\small Dedicated to Professor Olli Martio on the occasion of his 80th birthday celebration}}
\end{center}
\begin{abstract}
We study necessary and sufficient conditions for a Muckenhoupt weight $w \in L^1_{\loc}(\mathbb R^d)$ that yield almost sure existence of radial, and vertical, limits at infinity for Sobolev functions $u \in W^{1,p}_{\loc}(\mathbb R^d,w)$ with a $p$-integrable gradient $|\nabla u|\in L^p(\mathbb R^d,w)$. The question is shown to subtly depend on the sense in which the limit is taken.

First, we fully characterize the existence of radial limits. Second, we give essentially sharp sufficient conditions for the existence of vertical limits.  
%We also construct weights $w$ failing these conditions and functions $u$ in the weighted Sobolev space without vertical limits. 
In the specific setting of product and radial weights, we give if and only if statements. These generalize and give new proofs for results of Fefferman and Uspenski\u{\i}. 
\end{abstract}
%\todo{Average integral notation consistent?}
%{\color{red}
%\begin{enumerate}
%\item Edit Introduction
%\item Fill in proofs
%\item MSC: 46E36, ...
%\item add recent references in $\mathcal A_p$
%\end{enumerate}
%}
%{\color{blue}
%Conventions
%\begin{enumerate}
%\item Dimension $d$, $\mathcal R_p(w)$, $\mathcal A_p$, $\mathbb S^{d-1}, w(B)$
%\item How to indicate weight: $W^{1,p}(\R^d,w)$
%\item How to indicate Cubes and their lifts:
%\item $dx$?
%\item Letter $c$ for radial limit
%\end{enumerate}
%}

\section{Introduction}

\subsection{Overview}

The starting point for this paper is the following result on radial limits by Uspenski\u{\i} \cite{u1961}. 
If $1 \leq p < d$ and   $u:\R^d \to \R$ is a continuously differentiable function with a $p$-integrable gradient $|\nabla u|\in L^p(\mathbb R^d),$ then there exists a constant $c\in \mathbb R$ so that 
\begin{equation} \label{yradiaali}
\lim_{t\to \infty}u(t\xi)=c
\end{equation}
 for almost every $\xi$ in the unit sphere $\mathbb S^{d-1}.$ The requirement that $1\le p<d$ is necessary as seen by considering the function
%$u(x)=\int_0^{|x|}\frac {1}{(t+1)\log(1+t)}\, dt$ 
$u(x)=\log\log(2+|x|^2).$  This observation is credited to Timan \cite{t1975}.  

Let us say that a function $u$ has \emph{a unique almost sure  radial limit} if there is a finite value $c$ so that \eqref{yradiaali} holds for
almost every  $\xi \in \mathbb S^{d-1}.$ In more modern language, the statement above concerns precise representatives of functions in the  Sobolev space $\dot W^{1,p}(\mathbb R^d)$. This space consists of all locally $p$-integrable functions $u$ whose distributional gradient $\nabla u$ satisfies $\nabla u \in L^{p}(\mathbb R^d)$ (in the sense that $\partial_i u \in L^p(\mathbb R^d)$ for each $i=1,\dots, d$). Uspenski\u{\i}'s result then can be rephrased as saying that $1\le p<d$ if and only if every $u\in \dot W^{1,p}(\R^d)$ has a representative 
which 
has a unique almost sure  radial limit.

Besides radial limits, also vertical limits have been considered. A function $u:\R^d \to \R$ is said to have \emph{a unique almost sure  vertical limit} if for almost every $\overline{x} \in \R^{d-1}$ we have\footnote{Here, and in what follows, we identify $\R^d=\R^{d-1}\times \R$.}
\begin{equation} \label{vertikaali}
\lim_{t\to \infty}u(\overline{x},t)=c,
\end{equation}
where $c$ is independent of $\overline x$. For specific functions $u$ (such as $u(x,y)=\frac{x^4-y^2}{x^4+y^2+1}$), the constants $c$ in \eqref{yradiaali} and \eqref{vertikaali} may be different. However, when $|\nabla u|\in L^p(\mathbb R^d)$ and $1\le p<d$, they coincide. In \eqref{vertikaali}, we could also consider the limit $|t|\to \infty$, and assume that the limit almost surely equals $c$. Our discussion applies to this setting with few modifications; see Remark \ref{rmk:minfty}.

Indeed, Kudryavtsev had asked, if a Sobolev function would have unique almost sure  vertical limits. Fefferman \cite{f1974} and Portnov \cite{p1974} independently resolved this question and showed that under the same assumptions as for Uspenski\u{\i}, for $1\leq p < d$, as before, every $u\in \dot W^{1,p}(\R^d)$ has a unique almost sure vertical limit. 
%\begin{equation} \label{eqvertikaali}
%\lim_{t\to \infty}u(\overline{x},t)=c.
%\end{equation}
Further, the value of the almost sure vertical limit in \eqref{vertikaali} is the same as in the case of radial limits \eqref{yradiaali}. 

 The purpose of this paper is to study extensions of Uspenski\u{\i}'s, Fefferman's and Portnov's results to weighted Sobolev spaces; for unweighted generalizations see \cite{UV,Kstability}.
 %The assumption on the weights is driven by regularity considerations: Muckenhoupt weights; see equations \eqref{eq:muckenhoupt-p-1} and \eqref{eq:muckenhoupt} for the definitions.  
 Weighted Sobolev spaces, especially with Muckenhoupt weights, have played a crucial role in PDEs and the study of variational problems, starting from  \cite{FKS}.  They are still actively employed in these topics; see
 \cite{CavalNeumann,Tuoc,SK}. Weighted function spaces have been further studied by many authors in regard to their intrinsic properties, such as regularity and the existence of traces; see \cite{rychkov, shlykweighted, MR4044738,Tyulenev13,Tyulenev14, Tyulanev14B}. Further, especially Muckenhoupt weighted Sobolev spaces arise in non-linear potential theory and in analysis on metric spaces; see e.g. \cite{MR2180887, bjornbjorn, HKST07}. 
Indeed, the importance of Muckenhoupt weights can be gleaned from the extensive literature on the topic.

The choice of Muckenhoupt weights is driven in part by their regularity properties, and the fact that they have naturally appeared in various settings; see references above. Further, without some assumption on the weight, we would end up with issues regarding the precise representatives of Sobolev functions, and lack the required absolute continuity on generic lines; see for instance lemmas \ref{lem:sobolevinclusion} and \ref{lem:acclines}, which crucially use the Muckenhoupt assumption. This regularity theory is developed significantly in \cite{HKM06}. The class of Muckenhoupt weights $w\in\mathcal A_p$ is also natural to consider, since they guarantee a $p$-Poincar\'e inequality and doubling; see  equations  \eqref{doubling-1811} and \eqref{PI-1811} below. 

Our paper studies limits of weighted Sobolev functions. Our first results give characterizations for Muckenhoupt-weighted Sobolev functions to possess a unique almost sure  radial limit. Then, motivated by results of Fefferman and Portnov, we pursue the existence of \emph{vertical limits}. First, we note that the existence of vertical limits is \emph{more restrictive} than having radial limits. This is a phenomenon that already occurs with radial weights $w(x)=|x|^\alpha$ as will be shown in Remark \ref{ex:radialweight}. 

To obtain almost sure vertical limits, we will need to place a non-degeneracy assumption on the weight.  However, sufficient conditions prove more difficult and involve assumptions on regularity (integrability or a special structure). While in some settings these sufficient conditions also become necessary, in general there is a gap between them. Further, we provide examples to illustrate the partial sharpness of our conditions.
 
We take a small excursion to define notation. 
Throughout, we will only consider weights $w\in \A_p$, where $\A_p:=\mathcal A_p(\mathbb R^d)$ is the class of Muckenhoupt weights on $\mathbb R^d$. We will study notions with respect to the weighted Lebesgue measure $\mu$ with $d\mu=wdx$. Also, we will denote the weighted measure of a set $A \subset \R^d$ as $w(A)$. If $w\in \A_p$, then $w^{-\frac{1}{p-1}}\in L^{1}_{\loc}(\R^d)$ when $p>1$ (or $w^{-1}\in L^\infty_{\loc}(\R^d)$ when $p=1$) and it follows from H\"older's inequality that a function $u \in L^{p}_{\loc}(\R^d,w)$ also satisfies $u\in L^1_{\loc}(\R^d)$.

  We define $\dot W^{1,p}(\mathbb R^d,w)$ and $W^{1,p}_{\loc}(\R^d,w)$ to consist of all Lebesgue representatives of functions $u\in L^p_{\loc}(\R^d,w)$ so that $|\nabla u| \in L^p(\R^d,w)$ and $|\nabla u| \in L^p_{\loc}(\R^d,w)$, respectively.  The \emph{Lebesgue representative} is defined\footnote{It will be crucial for us, that Lebesgue representatives are defined with respect to the Lebesgue measure - and not with respect to the weighted measure $wdx$.} as $\tilde{u}(x)=\limsup_{r\to 0} \avint_{B(x,r)} u(y)\, dy$. Since this assumption is crucial for us, we highlight it here:
\begin{quote}
\textit{
For $w\in \A_p$ each $u\in W^{1,p}_{\loc}(\mathbb R^d,w)$ will be taken as its Lebesgue representative: $u(x)=\limsup_{r\to 0} \avint_{B(x,r)} u(y)\, dy$.}
\end{quote}
%{\color{red}Reference: Define a representative of Lesbesgue points: $u(x)=\limsup_{r\to 0} \dashint_{B(x,r)}ud\mu$. .}
For more details, see Subsection \ref{subsec:sobolev}. See also \cite[Section 4]{HKi98}.
  
Before we start the detailed discussion on our results, we present an example to illustrate the main results.

\begin{remark} \label{ex:radialweight} A useful family of examples in $\R^d$ to consider is the class of power weights $w(x)=|x|^\alpha$ for $\alpha \in \R$. 
\begin{enumerate}
\item If $\alpha < -d$, then $w$ is not locally integrable and $d\mu = w dx$ fails to give a locally finite Radon measure. If $\alpha \in (-d,d(p-1))$, then $w \in \A_p$. 
\item If $\alpha \in (-d,p-d]$, then the function $u(x)=\log\log(2+|x|^2)$ from above satisfies $u \in \dot W^{1,p}(\R^d,w)$ but does not have radial and vertical limits. 
\item If $\alpha \in (p-d,0)$, then Theorem \ref{remark-c} below gives the existence of a unique almost sure radial limit. However, vertical limits may fail to exist. Towards this, let $e_d=(0,\dots, 0, 1)$ be the unit vector in the $d$'th coordinate direction. There exists a function $u(x) \in \dot W^{1,p}(\R^d, w)$ for which the limit $\lim_{t\to \infty} u( x+te_d)$  exists for no $x$ with $|x|\leq 1/2$. Indeed, the function $u(x)=\sum_{i=1}^\infty \max\left\{1-|x-2^i e_d|,0\right\}$ is such a function. The reason behind this is that when $\alpha < 0$, the masses of unit sized cubes degenerate as the cubes move towards infinity.
\item Finally, if $\alpha \in [0,d(p-1))$, then both vertical and radial limits exist by Theorems \ref{remark-c} and \ref{thm:pekka-mizuta} below.
\end{enumerate}
\end{remark}

Next, we will present the results of this paper in more detail, starting with the radial setting and then proceeding to the vertical setting.

  \subsection{Radial limits} Our first theorem shows that the weak boundedness along a single ray, for all functions, will imply that a unique almost sure radial limit exists. In fact, the statement is even slightly stronger.

\begin{theorem}\label{thm:single limit} Let $w\in \A_p(\R^d)$ where $1\leq p<\infty$ and $d\geq 2$. Then the following two conditions are equivalent:

\begin{enumerate}
\item For every $u\in  \dot W^{1,p}(\R^d,w)$, there exists a $\xi \in \mathbb S^{d-1}$ so that $\liminf_{t\to \infty} |u(t\xi)| <\infty$.
\item Every $u\in  \dot W^{1,p}(\R^d,w)$ has a unique almost sure  radial limit.
\end{enumerate}
\end{theorem}

We highlight here the \emph{uniqueness} of the radial limit. In principle, one could consider the condition that the radial limit $\lim_{t\to\infty} u(t\xi)=c_\xi$ exists for a.e. $\xi \in \mathbb S^{d-1}$. \emph{A priori}, the limit $c_\xi$ could depend on the direction $\xi.$ However, it follows as a corollary to the theorem that, if the limits exist in this sense for \emph{every} function $u\in  \dot W^{1,p}(\R^d,w)$, then in fact the almost sure radial limit is independent of direction. The almost sure radial limit can further be computed in many average ways.

\begin{proposition}\label{thm:identify-c}
 Under either assumption of Theorem \ref{thm:single limit}, the unique almost sure radial limit $c$ satisfies each of the following three conditions:
 
 \begin{center}
\begin{tabular}{  p{5.2cm} | p{5.2cm} | p{5.5cm}  }
 $\displaystyle\lim_{r\to\infty}\avint_{\mathbb S^{d-1}}|u(r\xi)-c|d\mathcal H^{d-1}(\xi)=0$ & \hspace{.1cm}$\displaystyle\lim_{t\to\infty}\avint_{B(0,t)\setminus B(0,t/2)}|u-c|dx=0$ & \hspace{.1cm}$\displaystyle \lim_{|x|\to\infty} \avint_{B(x,|x|/2)}|u(y)-c|dy=0$ \\     
\end{tabular}
\end{center}
where $0$ is the origin of $\mathbb R^d$ and $B(x,r)$ is the ball with radius $r$ and center at $x$.
Further, the claim that for every $u\in \dot W^{1,p}(\R^d,w)$ there exists a constant $\tilde d$, so that any of these limits exists with $\tilde d$ replacing $c$, is equivalent with the conclusions of Theorem \ref{thm:single limit}.
\end{proposition}
\noindent Here $\avint_Afd\nu:=\frac{1}{\nu(A)}\int_Afd\nu$ for any given measure $\nu$, set $A$ with $\nu(A)>0$, and integrable function $f$ on $A$.
%This motivates one to ask if a similar result could also hold for weighted Sobolev spaces. Indeed, one can ask the same also for more general metric spacFfinaes $X$, such as the Heisenberg group. Towards this direction, one can state the following using the notion of $p$-a.e. coming from Modulus (see ...). {\color{red} Reformulate following for metric spaces?}

%By considering radial limits and employing the previous theorem, we get the following result which gives a compete answer in the case of a Muckenhoupt $\mathcal A_p$-weight. {\color{red} Additionally assume on $w$ to the finiteness of $\mathcal R_p$ is a condition at infinity} 
The crucial tool to prove these theorems is a quantity measuring the $p$-capacity at infinity; see e.g. \cite{HKM06} for the definition of capacity. Given
a locally integrable function $w$ with $w(x)>0$ for almost every $x\in \mathbb R^d,$ we define $w^s(A)=\left(\int_A w dx\right)^s$ when $A$ has strictly positive Lebesgue measure and $s\in \R$. We set 
\begin{equation}\label{eq:rpgtr1}
\Rp(w):=\sum_{i\in\mathbb N}(2^i)^{\frac{p}{p-1}}w^{\frac{1}{1-p}}(A_i) \text{\rm \ \ if $p>1$,}
%\int_{\mathbb{R}^d\setminus B(0,1)}|x|^{\frac{p(d-1)}{1-p}}w^{\frac{1}{1-p}}(x)dx(x) \text{\ \ if $p>1$, }
\end{equation}
and
\begin{equation}\label{eq:rpeq1}
\Ri(w):=\sup_{i\in\mathbb N}(2^iw^{-1}(A_i))
%\| |x|^{1-d}w^{-1}(x)\|_{L^\infty (\mathbb R^d\setminus B(0,1))}.
\end{equation}
where $A_i:=\{x\in\mathbb R^d:2^i\leq |x|<2^{i+1}\}$ for $i\in\mathbb N$.
%In what follows, we always assume that $|x|^{1-d}w^{-1}(x)$ belongs to $L_{\rm loc}^{p^*}(\mathbb R^d\setminus B(0,1))$ where $p^*$ is the H\"older conjugate of $p$ with $1\leq p<\infty$, to make sure that the finiteness of $\mathcal R_p(w)$ is a condition at infinity.

The finiteness of the quantity $\Rp(w)$, for $\mathcal A_p$-weights, actually characterizes when the family of curves $\gamma_\xi :[1,\infty) \to \R^d$ given by $\gamma_\xi(t) =  t\xi$ has positive $p$-modulus, but
neither this concept nor this result will be directly needed in this paper. We refer the reader to \cite{HKST07} for a discussion on modulus and to \cite{KN21} for further results. 
This phenomenon underlies the following theorem.

\begin{theorem}
	\label{remark-c} Let $w\in \mathcal A_p(\mathbb R^d)$ where $1\leq p<\infty$ and $d\geq 2$. Then the following two conditions are equivalent:
	\begin{enumerate}
		\item $\mathcal R_p(w)<\infty$.
		\item Every $u\in\dot W^{1,p}(\mathbb R^d,w)$ has a unique almost sure  radial limit.
		%\item For every $u\in\dot W^{1,p}(\mathbb R^d,w)$, there exists a constant $c\in\mathbb R$ only depending on $u$ such that 
		%\begin{equation}\notag
		%\lim_{r\to\infty}u(r\xi) \text{\rm \ \ exists \ \ and \ \ }\lim_{r\to\infty}u(r\xi)=c \text{\rm \ \ for $\mathcal H^{d-1}$-a.e $\xi\in\mathbb S^{d-1}$.}
	%	\end{equation}
	\end{enumerate}
	Moreover,   when either of these equivalent conditions is satisfied, 
		\[ \int_{\mathbb S^{d-1}}|u(r\xi)-c| d\mathcal H^{d-1}\lesssim \|\nabla u\|_{L^p(\mathbb R^{d}\setminus B(0,r),w)} \text{\rm  \ \ and\ \ }\avint_{B(0,r)\setminus B(0,r/2)}|u(x)-c|dx\lesssim \|\nabla u\|_{L^p(\mathbb R^{d}\setminus B(0,r/2),w)} 
		\]
for each $r>0$ and
	for every $u\in\dot W^{1,p}(\mathbb R^d,w)$.
	\end{theorem}

The above theorems  characterize the property of having unique almost sure radial limits. We refer the interested readers to \cite{KN21} for a version of this theorem on Carnot groups. We next turn our discussion to the case of weighted Sobolev spaces and vertical limits. 

\subsection{Vertical limits}

The example in Remark \ref{ex:radialweight} suggests that the existence of an almost sure vertical limit is a stronger property than the existence of a radial limit. Indeed, this is the case by the following argument. If radial limits fail to exist, then, by the proof of Theorem \ref{thm:single limit}, there exists a function $u$ with $\lim_{|x|\to \infty} u(x)=\infty$. Such a function fails to have any vertical limits.

However, even for radial weights $w$ with $\Rp(w)<\infty$, Remark \ref{ex:radialweight} together with Theorem \ref{remark-c} shows that vertical limits may fail to exist. The issue is that whenever one is able to construct cubes $(Q_i)_{i\in \N}$ marching off to infinity with $w(Q_i) \to 0$,  one can place ``bumps'' in them. Indeed, this yields a necessary condition for having almost sure vertical limits.

\begin{theorem}\label{thm:cubed-average} Let $w\in \A_p(\mathbb R^d)$ where {$1\leq p<d$} and $d\geq 2$. If every $u \in \dot W^{1,p}(\R^d,w)$ has a unique almost sure vertical limit, then $\Rp(w)<\infty$ and for every cube $Q \subset \R^{d-1}$ it holds that $\inf_{z\in \mathbb N} w(Q \times [z,z+1]) > 0$.
\end{theorem}
A more technical necessary condition will be seen in condition (1) of Theorem \ref{thm:pekka-mizuta-2nd}. The necessity of this condition is implied by Lemma \ref{lem:av-est} below.

\begin{remark}\label{rmk:p=1}
We remark briefly on the case of $p=1$. In this case Theorem \ref{thm:cubed-average} is an equivalence. If $\inf_{z\in \mathbb N} w(Q \times [z,z+1])>0$, then the fact that $w\in \A_1$ implies that there is a constant $c>0$ so that $w \geq c$ for a.e. $x\in 2Q \times [0,\infty)$. Let $\phi_Q$ be a  bounded $1$-Lipschitz function that equals to $1$ on $ Q\times [-1,\infty)$ and $0$ outside of $2Q \times [-1,\infty)$. Then, $\phi_Q u \in W^{1,1}(\R^d)$ and Fefferman's result implies that $\phi_Q u$ has a unique almost sure vertical limit. 
We also note that it is direct to verify that $\inf_{z\in \mathbb N} w(Q \times [z,z+1])>0$ implies that $\mathcal R_1(w)<\infty$. For this reason, in what follows, we will focus on the case $p>1$.
\end{remark}

Our next theorem gives a characterization of the existence of limits of certain averages. The condition is a slight strengthening of the one appearing in the previous theorem. Given a cube $Q\subset \mathbb R^d,$ we refer to its edge length by $\ell(Q).$ For a sequence of cubes $(Q_i)_{i\in\N}$, we define $Q_i \to \infty$ to mean that for every $R>0$, there exists an integer $N\in \N$ so that for $i \geq N$ we have $B(0,R) \cap Q_i = \emptyset$.

\begin{theorem}\label{lisatty}
Let $w\in \A_p(\mathbb R^d)$ where $1\le p<d$ and $d\ge 2.$ The following two conditions are equivalent:
\begin{enumerate} 
\item We have 
\begin{equation}\label{eq:lower-mass-bound}
\inf_{\ell(Q)=1} w(Q)>0.
\end{equation}
\item For every $u \in \Wdp(\R^d,w)$ there exists a unique constant $c$ so that for every sequence of cubes $Q_i \subset \R^{d}$ with $\liminf_{i\to\infty} \ell(Q_i)>0$ 
%$\limsup_{i\to\infty}\ell(Q_i)<\infty$, 
and $Q_i\to\infty$ it holds that

\[
\lim_{i\to\infty} \avint_{Q_i} u d\mu = c.
\]
\end{enumerate}
\end{theorem}

Indeed, the proof of the theorem will show that $c$ coincides with the unique almost sure radial limit. 
Even though we have a characterization  for the existence of limits of rough averages, the existence of vertical limits is more subtle. In order to move from the rough average limits in the above statement to vertical limits, 
one needs additional assumptions. To begin, consider an additional exponent $q\in [1,p]$, with $w \in \A_q$. Note that then $\A_q \subset \A_p$, and this is thus potentially a stronger requirement. For certain ranges of $q$ and $p$,  the existence of certain rough averages is equivalent to the existence of vertical limits.

\begin{theorem}\label{thm:vertical-lowerbound}
Let $w\in \mathcal A_q(\mathbb R^d)$ where $1\leq 
q \leq p<\infty$ and $d\geq 2$. Suppose that $\inf_{z\in\mathbb N}w(Q\times[z,z+1])>0$ for every cube $Q\subset\mathbb R^{d-1}$. If   
$qd-d+1<p<d$, 
then the following two conditions are equivalent:
\begin{enumerate}
\item For every $u\in \dot W^{1,p}(\R^d,w)$ there exists a constant $c\in \R$ for which $\lim_{z\to \infty, z\in\mathbb N} u_{Q\times [z,z+1]} = c$ for each cube $Q\subset \mathbb R^{d-1}$ of unit size.
\item Every $u\in \dot W^{1,p}(\mathbb R^d,w)$ has a unique almost sure vertical limit.
\end{enumerate}
%Here $u_A:=\avint_{A}ud\mu:=\frac{1}{w(A)}\int_Auwdx$ for each subset $A\subset\mathbb R^d$ with $w(A)>0$ and integrable function $u$ on $A$.
\end{theorem}

%\begin{remark}
%If $w\in \mathcal A_q$, then it always has doubling dimension $N=qd$ by \cite[Section 15.5]{HKM06}. Thus, if $qd-d+1 < p$, or $q < \frac{p-1}{d-1}$, then the condition holds. 
%\end{remark}

We remark that, by Theorem \ref{thm:cubed-average}, the assumption $\inf_{z\in\mathbb N}w(Q\times[z,z+1])>0$ is necessary for (2) to  hold. The main content here is that the existence of average limits, under weights of certain types, is equivalent to the existence of pointwise limits. 
From the proof, it also follows that the unique almost sure vertical limit of $u$ is the constant $c$ from the first condition.

The conclusion of the theorem is also true when $q>1$ and $p=qd-d+1$. Indeed, Muckenhoupt $\A_q$-weights with $q>1$ satisfy a self-improvement property: for every $w\in \A_q$ there exists an $\epsilon>0$ so that $w\in \A_{q-\epsilon}$; see \cite[Chapter V]{Stein}.
Thus, the previous result  fully characterizes the existence of almost sure vertical limits when $qd-d+1\le p<d$.  However, once 
%$p<N-d+1$, 
$p<qd-d+1$,
the question becomes more delicate: we are only able to give sufficient conditions for the existence of almost sure vertical limits. These take the form of either higher order integrability, or a special product structure for the measure.

\begin{theorem}\label{thm:pekka-mizuta} Let $1\leq q\leq \frac{p+d-1}{d}$ where $1<p<d$ and $d\geq 2$. If $w\in \mathcal \A_q(\mathbb R^d)$ satisfies $\inf_{\ell(Q)=1} w(Q)>0$,
then  every $u\in \dot W^{1,p}(\mathbb R^d,w)$ has a unique almost surely  vertical limit. Further, the vertical limit value equals the almost sure radial limit.
\end{theorem}

The range of exponents $q$ in the statement is sharp in the class of \emph{all weights}, but not necessarily for a given weight or subclass of weights; see Example \ref{ex:mainexample}.

\begin{remark}\label{remark1.14}  The assumptions of the previous theorems are related to each other by the following  implications for $w\in\mathcal A_p$. Firstly, when $p<d$
\[   \inf_{\ell(Q)=1} w(Q)>0 \Longrightarrow \mathcal R_p(w) < \infty.  \]
The proof of this is presented in the second paragraph of the proof of Theorem \ref{lisatty}. This implication can not be turned into an equivalence.

Secondly, condition (1) in Theorem \ref{thm:vertical-lowerbound} is equivalent to the condition that $$\sup_{t>0} \Rp(w_t) < \infty,$$
where $w_t:\R^d \to [0,\infty]$ is a translated weight defined by $w_t(\overline x,s)=w(\overline x,s-t)$. This claim follows from Lemmas \ref{lem:av-est} and \ref{Syl}. 

\end{remark}

In the final part of this introduction, we discuss sharp results for weights with a special structure.

\subsection{Special classes of weights}

First, we present a theorem for radial weights.

\begin{theorem}\label{thm:radial-weights} Let $w(x)=v(|x|)$ be a radial weight with $w \in \A_p(\R^d)$. Then the following two conditions are equivalent:

\begin{enumerate}
\item $\inf_{r>0} \int_{r}^{r+1} v(s) ds > 0.$
\item Every $u \in \dot W^{1,p}(\mathbb R^d,w)$ has a unique almost sure  vertical limit.
\end{enumerate}
\end{theorem}

Radial Muckenhoupt weights have been characterized in \cite{DMOSradial}. Indeed, a radial weight $w$ belongs to $\A_p(\R^d)$ if and only if $v_0(t)=v(t^{1/n})$ belongs to $\A_p(\R)$, where $w(x)=v(|x|)$ for each $x\in \R^d$.

Finally, we give a sharp result for those weights $w$ that have product structure: $w(\overline x,t)=w_1(\overline x)w_2(t)$. To state the theorem, we again need a translation invariant form of $\Rp(w).$ 

\begin{theorem}\label{thm:pekka-mizuta-2nd} Suppose that $1<p<d$. Let 
$w_1 \in \mathcal A_p(\R^{d-1}), w_2\in \mathcal A_p(\R)$ and $w(\overline{x},y)=w_1(\overline{x})w_2(y)$. Then the following two conditions are equivalent:

\begin{enumerate}
\item $\sup_{t>0}\mathcal R_p(w_t)<\infty$.
\item Every $u\in \dot W^{1,p}(\R^d, w)$ has a unique almost sure  vertical limit.
\end{enumerate}
\end{theorem}

\begin{remark} We note that, under these assumptions, $w(\overline{x},y) \in \A_p(\R^d)$. Indeed, this follows directly from the definition \eqref{eq:muckenhoupt-p-1} and \eqref{eq:muckenhoupt} using Fubini's theorem.
\end{remark}

\begin{remark} \label{rmk:minfty} In the case of vertical limits, for simplicity, we chose in Definition \ref{vertikaali} to only consider limits when $t\to \infty$. We could also consider the \emph{stronger} property, that the limit exists also as $t\to -\infty$, and that the value is (almost surely) the same. We could call this the bi-infinite unique almost sure vertical limit. Our theorems apply to this definition with few edits. Theorems \ref{thm:pekka-mizuta} and  \ref{thm:radial-weights} are symmetric with respect to the transformation $t\to -t$. In these theorems, the conditions about a vertical limit could be replaced with a bi-infinite unique almost sure vertical limit. Note that the value of the vertical limit $c$ coincides with the unique almost sure radial limit.

Other theorems are not quite symmetric with respect to the reflection of the $t$-axis, but they are easily modified to be such. These are Theorems \ref{thm:cubed-average}, \ref{thm:vertical-lowerbound} and \ref{thm:pekka-mizuta-2nd}. To obtain versions of them with bi-infinite limits replacing vertical limits, we need to perform the following simple modifications. In the first two, we replace $\N$ by $\Z$. In  the final statement, we substitute $\sup_{t\in \R} \Rp(w_t)$ for the supremum over $t>0$. The modified statements can be reduced to the original ones by using symmetry. As an example, consider Theorem \ref{thm:pekka-mizuta-2nd}. If $\sup_{t\in \R} \Rp(w_t)<\infty$, then both $\sup_{t>0} \Rp(w_t)<\infty$ and $\sup_{t<0} \Rp(w_t) < \infty$, and one can apply the original statement to conclude the existence of a limit when $t\to\infty$ and when $t\to -\infty$, which both coincide with the radial limit. For the converse direction, note that if $\sup_{t\in \R} \Rp(w_t)=\infty$, then either $\sup_{t>0} \Rp(w_t)=\infty$ or $\sup_{t<0} \Rp(w_t) = \infty$. In this case, there exists a function $u$ without a limit when $t\to\infty$ or when $t\to-\infty$.

\end{remark}

The organization of this paper is as follows. In Section \ref{sec:notation}, we recall notions of Sobolev spaces and their properties. In Section \ref{sec:radial-limit}, we discuss the case of radial limits and give proofs for Theorem \ref{thm:single limit}, Proposition \ref{thm:identify-c} and Theorem \ref{remark-c}. In Section \ref{sec:examples}, we give counter-examples. In Section \ref{sec:vertical-limit}, we discuss the case of vertical limits and give proofs for Theorems \ref{thm:cubed-average}-\ref{lisatty}-\ref{thm:vertical-lowerbound}-\ref{thm:pekka-mizuta}-\ref{thm:radial-weights}-\ref{thm:pekka-mizuta-2nd}.

\vspace{.5cm}
\noindent \textbf{Acknowledgments:}  The first author was supported by the Academy of Finland grant \# 345005. The second author and third author were supported by the Academy of Finland grant \# 323960. The first author thanks the Mathematics department at University of Jyv\"askyl\"a for a wonderful stay during Fall 2020 and early spring 2021, during which this research was started.\\ \\

\section{Notation and Preliminaries}\label{sec:notation}

\subsection{Metric and measure notions}

Throughout this paper, we employ the following conventions. The notation $A\lesssim B(A\gtrsim B)$ means that there is a constant $C$ only depending on the data such that $A\leq C\cdot B\ (A\geq C\cdot B)$, and $A\approx B$ means that both $A\lesssim B$ and $A\gtrsim B$. Where necessary, we write $A\lesssim_{a,b,c,\dots}B$, when a bound for $C$ depends on $a,b,c,\dots$.

We will consider only the $d$-dimensional Euclidean space $\R^d$, where $d\geq 2$, equipped with Euclidean distance and (absolutely continuous) measures $\mu$ given by $d\mu=wdx$ where $w\in L^1_{\loc}(\R^d)$
is non-negative. Such a $w$ will be called a weight. The norm of a vector $v\in \R^d$ is denoted by $|v|$. Points in $\R^d$ will either be denoted by $x\in \R^d$, or $r\xi\in\mathbb R^d$ where $r\in[0,\infty)$ and $\xi\in\mathbb S^{d-1}$, or $(\overline{x},t) \in \R^d$ where $\overline{x}\in \R^{d-1}$ and $t\in \R$. Under this notation, the direction corresponding to the last coordinate is called vertical.

The usual Lebesgue spaces with respect to the weight $w$ are denoted by $L^p(\mathbb R^d,w)$, for $p\in [1,\infty]$. We denote by $L^p_{\loc}(\R^d,w)$ and $L^p_{\loc}(\R^d)$ the spaces of locally $p$-integrable functions. Open balls with center $x_0$ and radius $r$ will be denoted $B(x_0,r)$. Given a set $A\subset \R^d$ we denote its Lebesgue measure and its weighted measure by $|A|$ (where from context it is evident that $A$ is not a vector in an Euclidean space) and $w(A),$ respectively. Further, when $|A|>0$ and $w(A)>0$, we denote
$$ \avint_A f dx := \frac{1}{|A|}\int_A f dx \text{\rm \ \ and \ \ } f_A:=\avint_{A}fd\mu:=\frac{1}{w(A)}\int_Afwdx$$
whenever the integral on the right-hand side is defined. Note that $f_A$ will only denote an average with respect to the weight $w$.

We will consider exponents $p\in [1,\infty)$ and we assume $w\in \A_p(\R^d)$, where $\A_p(\R^d)$ is the class of Muckenhoupt weights. Recall that, given $p \in (1,\infty),$ a weight $w$ belongs to $\A_p(\R^d)$ if $w>0$ a.e. and  if there is a constant $C \geq 1$ so that for every ball $B=B(x_0,r)$,

\begin{equation}\label{eq:muckenhoupt}
\left(  \avint_B w ~ dx  \right) \left( \avint_B w^{-\frac{1}{p-1}}dx\right)^{p-1} \leq C.
\end{equation}

If $p=1$, we write $w\in \A_1(\R^d)$ if $w>0$ a.e. and if there is a constant $C\geq 1$ so that for every ball $B=B(x_0,r)$ and a.e. $y\in B$

\begin{equation}\label{eq:muckenhoupt-p-1}
 \avint_B w~ dx \leq C w(y).
\end{equation}
The optimal constants $C$ in the statements are called  the $\A_p$-constants of the weights $w$. The above conditions will be referred to as the $\A_p$-conditions for the weights $w$.

We say that $w$ is doubling if there is a constant $c_D\geq 1$ such that  
\begin{equation}\label{doubling-1811}
w(B(x,2r))\leq c_Dw(B(x,r))
\end{equation}
for all balls $B(x,r)$. A weight $w$ supports a $p$-Poincar\'e inequality if there is a constant $c_P>0$ such that 
\begin{equation}\label{PI-1811}
\avint_{B(x,r)}|u-u_{B(x,r)}|d\mu\leq c_P r\left(\avint_{B(x,r)}|\nabla u|^pd\mu\right)^{\frac{1}{p}}
\end{equation}
for all balls $B(x,r)$ and for all locally integrable functions $u$ with locally integrable distributional derivative $\nabla u$. 

By Section 15 in \cite{HKM06}, we have the following theorem.
\begin{theorem}\label{thm:muckenhouppi}
Let $1\leq p<\infty$. If $w\in\mathcal A_p(\mathbb R^d),$ then $w$ is doubling  and supports a $p$-Poincar\'e inequality, namely \eqref{doubling-1811} and \eqref{PI-1811} hold for $w$ with constants that only depend on $p,d$ and the
$\A_p$-constant of $w.$

\end{theorem}
We refer interested readers to \cite{MR1336257,MR1379269,MR4044738,MR1816566,MR2180887,Stein}  for discussions on Muckenhoupt weights.

A curve $\gamma:I \to \R^d$ is a continuous mapping from an interval $I \subset \R$, where $I$ can be open, closed or infinite. A curve $\gamma:I\to \R^d$ is said to be an infinite curve, if $I=[0,\infty)$ and its length $\int_\gamma ds$ is infinite. A curve $\gamma$ is said to be locally rectifiable, if for every compact subset $J \subset I$ the curve $\gamma|_J$ is rectifiable.

\subsection{Sobolev spaces}\label{subsec:sobolev} Let $1\leq p<\infty$.
  We denote by $W^{1,p}_{\loc}(\R^d,w)$ the space of \emph{Lebesgue representatives} of functions $u\in L^{p}_{\loc}(\R^d,w)$ with distributional derivative $\nabla u \in L^p_{\loc}(\R^d,w)$. If $w=1$ we drop it from the notation. We remark that the Lebesgue representatives exist by the following simple lemma.
  
  \begin{lemma} \label{lem:sobolevinclusion} Suppose $p\in [1,\infty)$. If $u \in L^p_{\loc}(\R^d,w)$ with $|\nabla u|\in L^p_{\loc}(\R^d,w)$ and $w\in \A_p(\mathbb R^d)$, then $u\in L^1_{\loc}(\R^d)$ and $|\nabla u|\in L^1_{\loc}(\R^d)$. In particular, $u\in W^{1,1}_{\loc}(\R^d)$.
  \end{lemma}
  
  \begin{proof} Consider the case $p>1$. For any $g\in L^p_{\loc}(\R^d,w)$ and any ball $B=B(x,r)\subset \R^d$ we have from the definition in \eqref{eq:muckenhoupt} and H\"older's inequality that
  
 \begin{align*}
 \int_B |g| dx &\leq \left(\int_B |g|^p w dx \right)^{\frac{1}{p}} \left(\int_{B} w^{-\frac{1}{p-1}} dx \right)^{\frac{p-1}{p}} 
 \leq C^{\frac{1}{p}} \left(\int_B |g|^p w dx \right)^{\frac{1}{p}} \left(\int_{B} w dx \right)^{\frac{-1}{p}}|B|.
 \end{align*}
 
   Applying this to all balls, and $g=u$ and $g=|\nabla u|$ gives the claim for $p>1$. For $p=1$, the estimate follows similarly from the definition by using \eqref{eq:muckenhoupt-p-1}.
  \end{proof}

Let $\Wdp(\R^d,w)$ denote the space of (Lebesgue representatives of) $u\in W^{1,p}_{\loc}(\R^d,w)$ with $|\nabla u| \in L^p(\mathbb R^d,w)$.  The reason for choosing the Lebesgue representative is that then our function is \emph{absolutely continuous on almost every line}. While this fact is not novel, in the literature it only follows rather indirectly. Towards establishing absolute continuity in our setting, we introduce some further concepts.

The Hausdorff $(s,R)$-content of $E\subset\mathbb R^d$ is defined by 
\[
\mathcal H^s_R(E)=\inf\left\{ \sum_{i\in\mathbb N} r_i^s: E\subset \bigcup_{i\in\mathbb N}B_i \text{\rm \ and \ } r_i\leq R\right\}
\]
where $B_i$ are balls with radius $r_i$. The Hausdorff $s$-measure of $E\subset \mathbb R^d$ is $\mathcal H^s(E):=\lim_{r\to0}\mathcal H^{s}_r(E)$.

For $u \in L^1_{\loc}(\R^d)$ we define the set of \emph{non-Lebesgue points}
\[
NL_u = \{x \in \R^d : \lim_{r\to \infty} \frac{1}{|B(x,r)|}\int_{B(x,r)} u(y) dy  \text{ does not exist }\}.
\]
For $u\in \Wdp(\R^d,w)$ the set $NL_u$ of non-Lebesgue points is quite small. Indeed, one has the following result.

\begin{lemma}\label{lem:hausdorffcap}Suppose $1\le p<\infty$ and $w\in \A_p(\mathbb R^d)$. Let $u\in \Wdp(\R^d,w)$, and let $NL_u$ be the set of non-Lebesgue points of $u.$ Then $\mathcal{H}^{d-1}(NL_u)=0$.
\end{lemma}

\begin{proof}
By Lemma \ref{lem:sobolevinclusion}, we have $u\in W^{1,1}_{\loc}(\R^d)$. The claim then follows from \cite[Theorem 1 in 4.8 and Theorem 3 in 4.5.1]{evansgariepy}.
\end{proof}

  \begin{lemma}\label{lem:acclines}  Let $1\le p<\infty$ and let $w\in \A_p(\mathbb R^d).$ If $ u\in \Wdp(\R^d,w)$, then for a.e. $\overline{x}\in \R^{d-1}$, we have that the function $h: t \to h(t)=u(\overline{x},t)$ is absolutely continuous and $|h'|(t)\leq |\nabla u|(\overline{x},t)$ for almost every $t \in \R$.
  
  Further, for a.e. $\xi \in \mathbb S^{d-1}$  we have that $h(t) = u(t\xi):(0,\infty) \to \R$ is absolutely continuous with $|h'|(t) \leq |\nabla u|(t\xi)$ for a.e. $t\in (0,\infty)$.
  \end{lemma}
  
  \begin{remark} For the bound $|h'|(t)\leq |\nabla u|$, we need to fix an a.e. representative of $|\nabla u|$ for the claim. However, this choice only alters the null set removed.
  \end{remark}
  
  \begin{proof}[Proof of Lemma \ref{lem:acclines}]
  Since $w\in \A_p(\mathbb R^d)$, we have by Lemma \ref{lem:sobolevinclusion} that $u\in W^{1,1}_{\loc}(\R^d)$.  Now, the claim of absolute continuity on horizontal lines follows from \cite[Theorem 2 in 4.9.2] {evansgariepy}. We summarize the proof for the readers convenience and to indicate the small modifications needed. 
  
First, take any compactly supported radially symmetric smooth and non-negative function $\psi:\R^d \to [0,\infty)$, with $\int_{\mathbb R^d} \psi dx = 1$. Consider the mollified functions defined by $u_n \defeq u \star (n^{d}\psi(xn))$, where $\star$ denotes the convolution. For any Lebesgue point $x\in \R^d$ we have $u_n(x) \to u(x)$. This holds for $\mathcal{H}^{n-1}$-a.e. $x\in \R^d$. Thus, for almost every $\overline{x}\in \R^{d-1}$, we get $u_n(\overline{x},t) \to u(\overline{x},t)$ pointwise for all $t\in \R$. Denote by $\partial_{x_d}$ the partial derivative in the $d$'th direction. Then  $\partial_{x_d} u_n = (\partial_{x_d} u)\star (n^{d}\psi(xn))$. Up to passing to a subsequence and using Fubini's theorem, we have that for almost every $\overline{x}\in \R^{d-1}$ the functions $a_n(t) = \partial_{x_d} u_n(\overline{x},t)$ converge in $L^1(\R)$ to $a(t) = \partial_{x_d} u(\overline{x},t)$. These claims together give that, for a.e. $\overline{x}\in \R^{d-1}$, the absolutely continuous functions $h_n(t)=u_n(\overline{x},t)$ converge to an absolutely continuous function $h(t)=u(\overline{x},t)$ and that $h'(t)=a(t)=\partial_{x_d} u(\overline{x},t) \leq |\nabla u|(\overline{x},t)$ for a.e. $t$.

The same proof applies for radial curves by replacing Fubini with polar coordinates, and the derivative $\partial_{x_d}$ with a radial derivative.
  
\end{proof}

\subsection{Maximal functions}
The fractional maximal function of order $\alpha\ge 0$ of a locally integrable function $f$ at $x\in\mathbb R^d$ is defined by 
\[
\mathcal M_{\alpha,R} f(x):=\sup_{0<r<R} r^\alpha\avint_{B(x,r)}|f|d\mu
\] where $R\in(0,\infty]$.
Then $\mathcal M:=\mathcal M_{0,\infty}$ is the Hardy-Littlewood maximal function. Let us recall the weak Hardy–Littlewood  inequality, see for instance {\cite[Theorem 3.5.6]{HKST07}}.

\begin{theorem}\label{thm2.7} Let $1\leq p<\infty$. Suppose that $w\in\mathcal A_p(\mathbb R^d)$. Then there is a constant $C>0$ only depending on $w$ such that for $\lambda >0$,
\[
w(\{x\in \R^d: \mathcal Mf(x)>\lambda \})\leq \frac{C}{\lambda}\int_{\R^d}|f(x)|w(x)dx
\] for all  $f\in L^1_{\rm loc}(\mathbb R^d,w)$.
\end{theorem}

Let $w\in \mathcal A_p(\mathbb R^d)$. By \cite[Section 15.5]{HKM06}, there is a constant $C>0$ so that for all $x\in\mathbb R^d$ and $0<r<R$ we have 
\begin{equation}\label{mitanalaraja}
\frac{w(B(x,r))}{w(B(x,R))}\geq C\left(\frac{r}{R}\right)^{pd}.
\end{equation}

By applying {\cite[Lemma 2.6]{HKi98}} with this estimate, we obtain the following theorem.

\begin{theorem} \label{thm2.8}
Let $w\in \A_p(\mathbb R^d)$ and $0\le \alpha<pd.$ Suppose that $f\in L^1_{\rm loc}(\R^d,w)$  and let  $B$ be a bounded measurable set with $w(B)>0$. Then 
\[
\mathcal H^{pd-\alpha}_\infty (\{x\in B: \mathcal M_{\alpha, \text{\rm diam}(B)}f(x)>\lambda \})\leq \frac{C\text{\rm diam}^{pd}(B)w^{-1}(B)}{\lambda}\int_{\R^d} |f(x)|w(x)dx 
\]for $\lambda >0.$
Here $\text{\rm diam}(B)$ is the diameter of $B$ and $C$ depends only on $p,d,\alpha,$ and the constant in \eqref{mitanalaraja}.
\end{theorem}

We briefly summarize a useful chaining argument. (See e.g. \cite{MR1336257,sobolevmet} for an early and classical use of this method.) Suppose that $u \in W^{1,p}_{\rm loc}(\R^d,w)$, where $w\in \A_p(\mathbb R^d),$ and let $B(x,r)$ be a ball. 
Then, for every Lebesgue point (with respect to $\mu$) $y\in B(x,r)$ of $u$, we have the following. For the balls $B_i := B(y,2^{1-i}r)$ we have
\[
\left|u(y)-\frac{1}{w(B)} \int_B u d\mu\right| \leq |u_{B}-u_{B_0}| +\sum_{i=0}^\infty |u_{B_i}-u_{B_{i+1}}|.
\]
By applying the $p$-Poincar\'e inequality to each term and using the fact that $\sum_i (2^{1-i}r)^{1-\beta}\le Cr^{1-\beta}$ with a constant depending on $\beta$ when
$0\le \beta<1$  we conclude that there exists a constant $C$ depending only on the dimension $d,$ $p$ and the choice of $\beta$ so that 
\begin{equation}\label{pisteittainen}
\left|u(y)-\frac{1}{w(B)} \int_B u d\mu\right|^p \leq Cr^{p(1-\beta)} \mathcal M_{p\beta, \text{\rm diam}(B)}|\nabla u|^p(y).
\end{equation}
Consequently, this calculation together with the previous theorem gives the following standard estimate in the unweighted case $w=1$.

\begin{lemma} \label{lem:content-average} Let $s,p,q\geq 0$ be such that $d-s<q\leq d$. There exists a constant $C$ so that the following holds. If $u \in W^{1,q}_{\rm loc}(\R^d)$, $B=B(x,r) \subset \R^d$ and $\lambda>0$ then

$$\mathcal{H}^s_\infty\left(\left\{y \in B : \left|u(y)-\frac{1}{|B|} \int_B u dx\right|>\lambda\right\}\right) \leq C\frac{r^{q+s-d}}{\lambda^q} \left( \int_{2B} |\nabla u|^q dx \right)^{\frac{1}{q}}.$$
\end{lemma}

\section{Radial Limits}\label{sec:radial-limit}
In this section, we will prove Theorem \ref{thm:single limit}, Proposition \ref{thm:identify-c} and Theorem \ref{remark-c}.
\begin{lemma}\label{lem3.1}
Let $w\in \mathcal A_p(\mathbb R^d)$ where $1\leq p<\infty$ and $d\geq  2$.  If $\mathcal R_p(w)<\infty$, then for every $u\in \dot W^{1,p}(\mathbb R^d,w)$, there exists a constant $c\in  \mathbb{R}$ only depending on $u$  such that 
	\[\lim_{t\to\infty} u(t\xi) \text{ exists and equals to $c$} \text{\ \ \ for $\mathcal H^{d-1}$-a.e  $\xi\in\mathbb S^{d-1}$}.
	\]
\end{lemma}
\begin{proof}
Let $A_i:=B(0,2^{i+1})\setminus B(0,2^i)$, $i\in\mathbb N$.
We have $\avint_{A_i}w^{\frac{1}{1-p}}dx\approx \left(\avint_{A_i}wdx\right)^{1/(1-p)}\approx (2^i)^{\frac{-d}{1-p}}w^{\frac{1}{1-p}}(A_i)$ if $p>1$, and $\| w^{-1}\|_{L^\infty(A_i)}\approx \left(\avint_{A_i}wdx\right)^{-1}\approx (2^i)^dw^{-1}(A_i)$ because $w\in\mathcal A_p(\mathbb R^d)$. It follows that 
\[ \int_{\mathbb R^d\setminus B(0,1)}|x|^{\frac{p(d-1)}{1-p}}w^{\frac{1}{1-p}}(x)dx \approx\sum_{i\in\mathbb N}(2^i)^{\frac{p(d-1)}{1-p}+d}\avint_{A_i}w^{\frac{1}{1-p}}(x)dx\approx \sum_{i\in\mathbb N}(2^i)^{\frac{p}{p-1}}w^{\frac{1}{1-p}}(A_i)=\mathcal R_p(w)
\] if $p>1$, 
and
\[\| |x|^{1-d}w^{-1}(x)\|_{L^\infty(\mathbb R^d\setminus B(0,1))}=\sup_{i\in\mathbb N} \| |x|^{1-d}w^{-1}(x)\|_{L^\infty(A_i)}\approx\sup_{i\in\mathbb N} 2^iw^{-1}(A_i)=\mathcal R_1(w).
\]
By the H\"older inequality, we obtain from these estimates  that
\begin{equation}
\label{eq:3.2-2409}\int_{\mathbb S^{d-1}}\int_1^\infty |\nabla u|(r\xi)drd\mathcal H^{d-1}(\xi) \lesssim \|\nabla u\|_{L^p(\mathbb R^d\setminus B(0,1),w)} \max\{\mathcal R_{p}(w)^{\frac{p-1}{p}}, \mathcal R_1(w)\}.
\end{equation}
Our assumption that $\mathcal R_p(w)<\infty$ implies that the right-hand side is finite.
%;  it is easy to check from the definitions that $\mathcal R_1(w)<\infty$ if $\mathcal R_1(w)<\infty$ for some $p>1.$	
Hence, by the Fubini theorem, it follows that $\int_1^\infty|\nabla u|(r\xi)dr<\infty$  for $\mathcal H^{d-1}$-a.e $\xi\in\mathbb S^{d-1}.$ Consequently,
$\lim_{r\to\infty}u(r\xi)$ exists for $\mathcal H^{d-1}$-a.e $\xi\in\mathbb S^{d-1}$ because {$u$ is absolutely continuous for a.e radial curve by Lemma \ref{lem:acclines}}. 

It suffices to show the uniqueness of $\lim_{t\to\infty}u(t\xi)$ for $\mathcal H^{d-1}$-a.e $\xi\in\mathbb S^{d-1}$. We argue by contradiction and assume that two different limits are attained through two subsets of
$\mathbb S^{d-1}$ of positive measure. By a simple measure theoretic argument, adding a suitable constant to $u$ and finally by multiplying $c$ by  another suitable constant, we may assume that there are subsets $E$ and $F$ of  $\mathbb S^{d-1}$ such that 
\begin{equation}
\label{eq:3.1-1109}
\begin{cases}
\mathcal H^{d-1}(E)\geq \delta, \mathcal H^{d-1}(F)\geq \delta,\\
u(r\xi)\geq 1 \text{\rm\ \ for all $r\geq r_0$ and $\xi\in E$}, \\
 u (r\xi)\leq 0 \text{\rm \ \ for all $r\geq r_0$ and $\xi\in F$}
\end{cases}
\end{equation}
for some $\delta >0$ and some $r_0<\infty.$
Let $j\in \mathbb N$ with $2^j\ge r_0.$ We define $E_j= \{(r\xi): r\in[2^j,2^{j+1}), \xi\in E \}$ and $F_j= \{(r\xi): r\in[2^j,2^{j+1}), \xi\in F \}$.
Obviously, $u|_{E_j}\geq 1$ and $u|_{F_j}\leq 0$. We split our argument into two cases depending on whether or not there are points $x$ in $E_j$ and $y$ in $F_j$ so that neither $|u(x)-u_{B(x,2^{j-2})}|$ nor $|u(y)-u_{B(y,2^{j-2})}|$ exceeds $1/5$. If such points can be found, then $1\leq |u(x)-u(y)| \leq 1/5+|u_{B(x,2^{j-2})}-u_{B(y,2^{j-2})}|+1/5$ and hence $\frac 3 5 \leq |u_{B(x,2^{j-2})}-u_{B(y,2^{j-2})}|$.  One can clearly find balls $\{B_i\}_{i=1}^M$ with radius $2^{j-2}$ and center in $B(0,2^{j+1})\setminus B(0,2^j)$, with $M$ only depending on $d$, such that $B_1= B(x,2^{j-2})$, $B_M=B(y,2^{j-2})$, and $B_i\cap B_{i+1}$ contains a ball with radius $2^{j-2}/100$. By doubling and the $p$-Poincar\'e inequality, it follows that 
\begin{align}
\frac 3 5 \leq &|u_{B(x,2^{j-2})}-u_{B(y,2^{j-2})}| \lesssim \sum_{i=1}^M 2^{j-2} \left ( \avint_{B_i}|\nabla u|^pd\mu \right)^{\frac{1}{p}} 
\lesssim  {2^j}  \left(\avint_{B(0,2^{j+2})\setminus B(0,2^{j-1})} |\nabla u|^pd\mu\right )^{\frac{1}{p}}.\label{eq:3.2-1109}
\end{align}
The second alternative, by symmetry, is that for all points $x$ in $E_j$ we have that $1/5\leq |u(x)-u_{B(x,2^{j-2})}|$. 
Since almost every $x$ is a Lebesgue point (with respect to $\mu$) of $u$ by the Lebesgue differentiation theorem, we have by \eqref{pisteittainen} the estimate 
$$1/5\leq |u(x)-u_{B(x,2^{j-2})}|\lesssim 2^{j-2}\mathcal M_{0,2^{j-1}}^{1/p}|\nabla u|^p(x).$$
By Theorem \ref{thm2.7} applied to the zero extension of $|\nabla u|^p$ to the exterior of $B(0,2^{j+2})\setminus B(0,2^{j-1})$, we obtain that 
$$w(E_j)\leq  C2^{jp}\int_{B(0,2^{j+2})\setminus B(0,2^{j-1})}|\nabla u|^pd\mu.$$ Combining this with \eqref{eq:3.2-1109} gives 
\[
\min\{ w(E_j), w(F_j)\} \lesssim 2^{jp} \int_{B(0,2^{j+2})\setminus B(0,2^{j-1})}|\nabla u|^pd\mu.
\]
 Analogously to the argument for \eqref{eq:3.2-2409}, H\"older's inequality together with $\mathcal R_p(w)<\infty$ yields, for all $j$ with $2^j\ge r_0,$ the estimates
\[
2^j\mathcal H^{d-1}(E)=\int_{E}\int_{2^j}^{2^{j+1}} drd\mathcal H^{d-1}\lesssim  w^{1/p}(E_j) \text{\rm \ \ and \ \ } 2^j\mathcal H^{d-1}(F)=\int_{E}\int_{2^j}^{2^{j+1}} drd\mathcal H^{d-1}\lesssim  w^{1/p}(F_j).
\]
Therefore, we obtain that
\[
\min \{ (\mathcal H^{d-1}(E))^p, (\mathcal H^{d-1}(F))^p \} \lesssim  \int_{B(0,2^{j+2})\setminus B(0,2^{j-1})}|\nabla u|^pd\mu \to 0 \text{\rm \ \ as  $j\to\infty$}
\] which contradicts \eqref{eq:3.1-1109}. The claim follows.
\end{proof}

\begin{lemma}\label{lem3.3} Under the assumption of Lemma \ref{lem3.1}, the constant $c$ satisfies both
		\[\int_{\mathbb S^{d-1}}|u(r\xi)-c| d\mathcal H^{d-1}(\xi)\lesssim \|\nabla u\|_{L^p(\mathbb R^{d}\setminus B(0,r),w)} \text{\rm  \ \ and\ \ }\avint_{B(0,r)\setminus B(0,r/2)}|u(x)-c|dx\lesssim \|\nabla u\|_{L^p(\mathbb R^{d}\setminus B(0,r/2),w)} 
		\]
for each $r>0$ and
	for every $u\in\dot W^{1,p}(\mathbb R^d,w)$.
\end{lemma}
\begin{proof}Since $u$ is absolutely continuous on almost every radial line by Lemma \ref{lem:acclines}, 
 inequality \eqref{eq:3.2-2409} yields that, for each $r>0$,
	\[
	\int_{\mathbb S^{d-1}}|u(r\xi)-c|d\mathcal H^{d-1}(\xi)\leq \int_{\mathbb S^{d-1}}\int_r^\infty |\nabla u|(r\xi)drd\mathcal H^{d-1}(\xi) \leq \|\nabla u\|_{L^p(\mathbb R^d\setminus B(0,r),w)} \mathcal R_{p}(w)^{\frac{p-1}{p}},
	\]
	if $p>1,$ and one obtains the same bound with $\mathcal R_1(w)$ replacing $\mathcal R_{p}(w)^{\frac{p-1}{p}}$ if $p=1$.
	It follows that for each $r>0$,
		\begin{align*}
		\avint_{B(0,r)\setminus B(0,r/2)}|u(x)-c|dx\leq &\frac{1}{{|B(0,r)\setminus B(0,r/2)|}}\left(\int_{r/2}^rs^{d-1}ds\right) \sup_{r/2\leq s\leq r}\int_{\mathbb S^{d-1}}|u(s\xi)-c|d\mathcal H^{d-1}(\xi)\\
		\lesssim& \sup_{r/2\leq s\leq r}\int_{\mathbb S^{d-1}}|u(s\xi)-c|d\mathcal H^{d-1}(\xi)\lesssim  \|\nabla u\|_{L^p(\mathbb R^{d}\setminus B(0,r/2),w)}.
		\end{align*}
		The claim follows.
\end{proof}

By the doubling property of Muckenhoupt weights, the estimates from  Lemma \ref{lem3.3} yield the following corollary.

\begin{corollary}\label{cor3.4}
Let $1\leq p<\infty$ and let $w\in \A_p(\mathbb R^d).$ If $\mathcal R_p(w)<\infty,$ then for every $u\in \dot W^{1,p}(\mathbb R^d,w)$ there exists a constant $c$ such that 
\[\lim_{r\to\infty}\avint_{\mathbb S^{d-1}}|u(r\xi)-c|d\mathcal H^{d-1}(\xi)=\lim_{t\to\infty}\avint_{B(0,t)\setminus B(0,t/2)}|u(x)-c|dx=\lim_{|x|\to\infty} \avint_{B(x,|x|/2)}|u(y)-c|dy=0\]
and
\[
\lim_{r\to\infty}\avint_{\mathbb S^{d-1}}u(r\xi)d\mathcal H^{d-1}(\xi)=\lim_{t\to\infty}\avint_{B(0,t)\setminus B(0,t/2)}u(x)dx=\lim_{|x|\to\infty} \avint_{B(x,|x|/2)}u(y)dy= c.
\]
\end{corollary}

\begin{lemma}\label{lem3.5} Let $1\leq p<\infty$ and let $w\in \A_p(\mathbb R^d).$
If $\mathcal R_p(w)=\infty$, then there exists $u\in\dot W^{1,p}(\mathbb R^d,w)$ such that $\lim_{|x|\to\infty}u(x)\equiv \infty$.
\end{lemma}
\begin{proof} Let $A_i:=B(0,2^{i+1})\setminus B(0,2^i)$, $i\in\mathbb N$.

Since $\mathcal R_p(w)=\infty$, depending on the value of $p,$ there exists a sequence $\{a_k\}_{k\in\mathbb N}$ or $\{b_k\}_{k\in\mathbb N}$ with $a_k<a_{k+1}, b_k<b_{k+1}$, $\lim_{k\to\infty}{a_k}=\lim_{k\to\infty}{b_k}=\infty$ such that 
\begin{equation}\label{eq-3.4}
\sum_{i=a_k}^{a_{k+1}}(2^i)^{\frac{p}{p-1}}w^{\frac{1}{1-p}}(A_i)>2^k \text{\rm \ \ if $p>1$\ \ and \ \ } 2^{b_k}w^{-1}(A_{b_k})>2^k \text{\rm \ \ if $p=1$.}
\end{equation}

	Let
	\begin{equation}\notag
	\label{eq4.7}  g_p(x)=\sum_{k=1}^\infty \left (\sum_{i=a_{k}}^{{a_{k+1}}}\frac{(2^i)^{\frac{1}{p-1}}w^{\frac{1}{1-p}}(A_{i})}{\sum_{i=a_k}^{a_{k+1}} (2^i)^{\frac{p}{p-1}}w^{\frac{1}{1-p}}(A_{i}) } \chi_{A_{i}}(x)\right ) \text{\rm \  if $p>1$, and\ }
	g_1(x)=\sum_{k=1}^\infty 2^{-b_k}\chi_{A_{{b_k}}}(x).
	\end{equation}
	We define
	$u(x):=\inf \int_{\gamma_{0,x}} g_pds \text{\ \ \rm for $x\in \mathbb R^d$}
	$
	where the infimum is taken over all rectifiable curves $\gamma_{0,x}$ connecting the origin $0$ and $x$. Then $u$ is locally Lipschitz and $|\nabla u|\leq g_p$ almost everywhere  with respect to the Lebesgue measure and consequently also $\mu$-a.e. Let $N$ be arbitrary.
	By a similar argument as in \cite{KN21, MR4344318,KNW22},
	 we have that for all $x\in  \mathbb R^d$ with $|x|=N$,
	\begin{align*}
	u(x)=& \inf_{\gamma_{0,x}} \sum_{2^{a_{k+1}}\leq N}\sum_{i=a_k}^{a_{k+1}} \frac{(2^i)^{\frac{1}{p-1}}w^{\frac{1}{1-p}}(A_{i})}{\sum_{i=a_k}^{a_{k+1}} (2^i)^{\frac{p}{p-1}}w^{\frac{1}{1-p}}(A_{i}) } \int_{\gamma_{0,x}\cap A_{i}}ds 
	\gtrsim&  \sum_{2^{a_{k+1}}\leq N}\sum_{i=a_k}^{a_{k+1}} \frac{(2^i)^{\frac{1}{p-1}}w^{\frac{1}{1-p}}(A_{i})}{\sum_{i=a_k}^{a_{k+1}} (2^i)^{\frac{p}{p-1}}w^{\frac{1}{1-p}}(A_{i}) } 2^i
	\gtrsim N\end{align*} if $p>1$,
	and that
	$
	u(x)=\inf_{\gamma_{0,x}}\int_{\gamma_{0,x}}g_1ds = \inf_{\gamma_{0,x}} \sum_{2^{b_k}\leq  N} 2^{-b_k}\int_{\gamma_{0,x}\cap A_{{b_k}}}ds 
	\gtrsim  \sum_{2^{b_k}\leq N}2^{-b_k} 2^{b_k}\gtrsim N.
	$
	Hence
	$\lim_{|x|\to \infty}u(x)=\infty$. It suffices to prove that  $g_p\in L^p(\mathbb R^d,w)$. Using \eqref{eq-3.4}, we have that 
		\begin{equation}
	 \int_{\mathbb R^d}g_p^pd\mu =\sum_{k=1}^\infty \sum_{i=a_{k}}^{{a_{k+1}}}  \int_{A_{i}}\left (\frac{(2^i)^{\frac{1}{p-1}}w^{\frac{1}{1-p}}(A_{i})}{\sum_{i=a_k}^{a_{k+1}} (2^i)^{\frac{p}{p-1}}w^{\frac{1}{1-p}}(A_{i})}\right )^p d\mu\notag 
	= \sum_{k=1}^\infty \frac{1}{ \left (\sum_{i=a_k}^{a_{k+1}} (2^i)^{\frac{p}{p-1}}w^{\frac{1}{1-p}}(A_{i}) \right )^{p-1}} \leq \sum_{k=1}^\infty\frac{1}{2^{k(p-1)}} \notag
	\end{equation}if $p>1$, 
	and that
	$
	 \int_{\mathbb R^d}g_1d\mu =\sum_{k=1}^\infty \int_{A_{{b_k}}} 2^{-b_k}d\mu =\sum_{k=1}2^{-b_{k}}w(A_{{b_k}})\leq \sum_{k=1}^\infty\frac{1}{2^{k}}.\notag
	$ Then $g_p\in L^{p}(\mathbb R^d,w)$.
	The claim follows.
\end{proof}
\begin{proof}[Proof of Theorem \ref{thm:single limit}] The implication
$(1)\Rightarrow (2)$ is given by Theorem \ref{remark-c} and Lemma \ref{lem3.5}.
Furthermore, the implication $(2)\Rightarrow (1)$ is trivial. \end{proof}

\begin{proof}[Proof of Proposition \ref{thm:identify-c}] The claim follows because the existence of each of these limits is equivalent to $\mathcal R_p(w)<\infty$ by Corollary \ref{cor3.4} and Lemma \ref{lem3.5}.
%The last claim is given by Corollary \ref{cor3.4}.
\end{proof}

\begin{proof}[Proof of Theorem \ref{remark-c}] The implication $(1)\Rightarrow (2)$ is given by Lemma \ref{lem3.1} and the implication 
$(2)\Rightarrow (1)$ by Lemma \ref{lem3.5}.

The last claim is given by Lemma \ref{lem3.3}.
\end{proof}

\section{Counter-examples}\label{sec:examples}

Our counter-examples will involve the construction of certain bump-functions. We will need the following explicit $\A_p$-weights.

\begin{example}\label{ex:Apweight} Let $q,p\in [1,d)$ with $q\leq p$. Further fix $\alpha \in [0,(d-1)(q-1)),\beta\in [0,d-p)$ when $q>1$ and let $\alpha=0, \beta\in [0,d-1)$ for $q=1$. Set

\[
w(\overline{x},t) = \begin{cases} 2^{-(\alpha+\beta)i-1}(1+|\overline{x}|^\alpha) & \text{if\ \ }2^i \leq t \leq 2^{i+1}, |\overline{x}|\leq 2^{i}, i\in\mathbb N\bigcup\{0\} ; \\
\min\{|(\overline {x},t)|^{-\beta},1\} & \text{otherwise.} \end{cases}
\]

Then $w\in \A_q(\R^d)$ and $\Rp(w)<\infty$.
\end{example}

\begin{proof} 
Since $0< w(\overline {x},t) \leq 1$, we have that $w\in L^1_{\loc}(\R^d)$. Fix a ball $B=B((\overline{y},s),r) \subset \R^d$. 
Since the necessary computations in what follows are rather technical, we only sketch the main points and leave the details to the reader.
A simple case study gives

\[
\avint_B w dx  \leq \sup_{(\overline{x},t) \in B} w(\overline{x},t)  \lesssim \begin{cases} \min(1,r^{-\beta}) & \text{if \ \ }r\geq |(\overline{y},s)|/2;\\ 
\min(1,|(\overline{y},s)|^{-\beta}) & \text{if \ \ } r\leq |(\overline{y},s)|/2, s\leq 1 ;\\
|(\overline{y},s)|^{-\beta}  & \text{if \ \ } r\leq |(\overline{y},s)|/2, |\overline{y}|\geq s, s\geq 1; \\
|(\overline{y},s)|^{-\beta} s^{-\alpha}(|\overline{y}|+r)^\alpha & \text{if \ \ } r\leq |(\overline{y},s)|/2, |\overline{y}|\leq s, s\geq 1. \\
    \end{cases}
\]

We continue by estimating
\[
I = \left(\frac{1}{|B|} \int_B w^{-\frac{1}{q-1}}(x) dx\right)^{q-1}
\]
in the case $1<p<d.$
Again, one applies a case study. We begin with some pointwise estimates for $w^{-\frac{1}{q-1}}$. 
\begin{enumerate}
\item[\textbf{A:}] When $|(\overline{x},t)|\leq 1$, we have $w(\overline{x},t)=w^{-\frac{1}{q-1}}(\overline x,t) = 1$.
\item[\textbf{B:}] 
When $|\overline{x}|\leq t$ and $t\in [2^i,2^{i+1}]$ for some $i\geq 0$, we use the bound $w(\overline{x},t) \gtrsim 2|\overline{x}|^{\alpha}t^{-\beta-\alpha}$. 
\item[\textbf{C:}]  When $|\overline{x}|\geq t$ and $t\geq 2$ or when $t\in [0,2]$, we use the bound $w(\overline{x},t) = \min\{|(\overline{x},t)|^{-\beta},1\}$.
\end{enumerate}

We consider four different cases to estimate $I$, depending on the location of the center $(\overline{y},s)$ and the radius $r$ of the ball $B$.

\begin{enumerate}
\item \textbf{If $r\geq |(\overline{y},s)|/2$}, then  $B((\overline{y},s),r) \subset B((\overline{0},0), 4r).$ To estimate $I$ from above, it suffices to replace $B((\overline{y},s),r)$ with  $B((\overline{0},0), 4r).$ Divide the integration over $B((\overline{0},0),4r)$ to regions where A,B,C apply. 
Observe that
\[
\int_{\{\overline{x}\in \R^{d-1}: r/2 \leq |\overline {x}|\leq r\}} |\overline{x}|^{\frac{-\alpha}{q-1}} d\overline x \lesssim r^{d-1-\frac{\alpha}{q-1}}
\]
holds whenever $\alpha \in [0,(d-1)(q-1))$. Consequently, by Fubini's theorem, whenever $r>0$ 
\[
\int_{B(0,4r)} w^{\frac{-\alpha}{q-1}}(\overline{x},t) dx \lesssim r^{d-\frac{\alpha}{q-1}+\frac{\beta+\alpha}{q-1}}.
\]
This bound together with the bounds from A,B,C can be used to conclude that $I\lesssim \max(1,r^\beta).$
\item \textbf{The case $r\leq |(\overline{y},s)|/2$ and $s\leq 1$:} From the definition of $w$ and the bound C one obtains $\inf_{(\overline{x},t) \in B} w(\overline{x},t) \gtrsim \min\{1,|(\overline{y},s)|^{-\beta}\}$. Thus, $I \lesssim \max\{1,|(\overline{y},s)|^\beta\}$.
\item \textbf{If $r\leq |(\overline{y},s)|/2$, $|\overline{y}|\geq s$ and $s\geq 1$:} In this case one can use bounds B and C to show that $\inf_{(\overline{x},t) \in B} w(\overline{x},t) \gtrsim \min\{|(\overline{y},s)|^{-\beta},1\}$. Thus, $I \lesssim |(\overline{y},s)|^\beta$.
\item \textbf{If $r\leq |(\overline{y},s)|/2$ and $|\overline{y}|\leq s$ and $s\geq 1$:} Divide the integral $\int_B w^{-\frac{1}{q-1}}(\overline{x},t) dx$ to integrals over the regions with $2^{k}\leq t < 2^{k+1}$ for $k\in \N$, and possibly a portion with $|t| \leq 1$. For the first set use estimate A and for the latter use C. Integrating, with a similar bound as in case (1), and adding the obtained bounds again yields the estimate $I \lesssim |(\overline{y},s)|^\beta$.
\end{enumerate}

By combing the cases (1)--(4) with the estimate on the integral average of $w$ from the beginning of our proof we conclude that $w\in \A_q(\mathbb R^d)$. Towards $\Rp(w)<\infty$, let $A_k = \{(x,t) : 2^k < |(\overline{x},t)|\leq 2^{k+1}\}$. We have that $w(A_k) \gtrsim 2^{k(d-\beta)}.$ Hence

\begin{align*}
\Rp(w)  &= \sum_{k\in\mathbb N}(2^k)^{\frac{p}{p-1}}w^{\frac{1}{1-p}}(A_k)   
\lesssim \sum_{k=0}^\infty 2^{\frac{kp}{p-1}} 2^{\frac{(\beta-d)k}{p-1}} <\infty
\end{align*}
since $\beta < d-p$.

When $p=1$, we must have $\alpha=0$, and it is direct to establish the inequality in \eqref{eq:muckenhoupt-p-1} by estimating the minimimum of $w$ in $B$. The estimate for $\Ri$ follows directly by the restriction $\beta \in [0,d-1)$ and the definition in \eqref{eq:rpeq1}.
\end{proof}

Using these weights we can find examples of weighted Sobolev functions, which lack vertical limits -- even in a rough and average sense. 
In what follows, cubes in $\R^d$ will be written as $Q=Q(x,\ell(Q))=\prod_{i=1}^d [x_i-\ell(Q)/2,x_i+\ell(Q)/2]$, 
where $\ell(Q)>0$ is the edge length of $Q$ and $x=(x_1,\dots, x_d) \in \R^d$ is the center of the cube. We will say the cube is centered at $x$. 
If $Q=Q(x,\ell(Q))$, then, for $a>0$, $a Q=Q(x,a\ell(Q))$ is the cube with the same center and of 
edge length $a\ell(Q)$.
%Furthermore, if $Q \subset \R^{d-1}$ is a cube, we 
%define the lifted cube $Q(t) \subset \R^d$ by setting $Q(t) = Q \times [t,t+\ell(Q)].$  
 If $(Q_i)_{i\in\mathbb N}$ is a sequence of cubes, we say that $Q_i$ go to infinity, or $Q_i \to \infty$, if for any $R>0$, there is a $N\in\mathbb N$ so that $Q_i \cap B(0,R) = \emptyset$ for all $i\geq N$.

We introduce a bump-function associated to a cube $Q.$
Let $\psi:\R\to \R$ be a function given by $\psi(x)=\min(1,\max(0,1-2|x|))$. Given a cube $Q=Q(x,\ell(Q)),$  define $\psi_Q(y) = \prod_{i=1}^d \psi((y_i-x_i)/\ell(Q)),$  where $y=(y_1,\dots, y_d) \in \R^d$. Then $\psi_Q$  is $2d\ell(Q)^{-1}$-Lipschitz, 
$\psi_Q(x)=0$ for $x\not\in Q$ and $\psi_Q(x)=1$ for $x\in \frac{1}{2}Q$.

%Occasionally, we will specify a cube by its center and its edge length. Then we write $Q(x,\ell) = \prod_{i=1}^d [x_i-\ell/2,x_i+\ell/2]$, where $x=(x_1,\dots, x_d) \in \R^d$ is the center and $\ell>0$ is the side length. 

\begin{example}\label{example:rough}  Suppose that $p\in[1,\infty)$. There exists a weight $w\in \A_p(\mathbb R^d)$ with $\Rp(w)<\infty$ and a function $u\in W^{1,p}(\R^d,w)$ and a sequence of cubes $(Q_i)_{i\in \N}$ with $Q_i \to \infty$, $\liminf_{i\to\infty} \ell(Q_i) >0,$ so that $\lim_{i\to\infty} u_{Q_i}$ does not exist. 

Indeed if $w \in \A_p(\mathbb R^d)$ is a weight for which there exists a sequence of cubes $Q_i$ with
\begin{enumerate}
\item $Q_i \to \infty,$ 
\item $\liminf_{i\to\infty} \ell(Q_i) >0,$
\item $\liminf_{i\to \infty} w(Q_i) = 0$,
\end{enumerate}
then, there exists $u\in W^{1,p}(\R^d,w)$ so that $\lim_{i\to\infty} u_{Q_i}$ does not exist. Further, if we have that $Q_i=Q\times [n_i,n_i+1]$ for some increasing sequence $(n_i)_{i\in \N}$, then $\lim_{t\to\infty} u(\overline{x},t)$ does not exist for any $\overline {x}\in Q$.

\end{example}

\begin{proof} 
We first prove the second claim. Assume that a weight $w\in \A_p$ and cubes $Q_i$, with $i\in\N$, exist as in the claim. 

  Pick a $\delta$ with $0<\delta < \liminf_{i\to\infty} \ell(Q_i)$. By passing to a subsequence, we may assume that $w(Q_i) \leq \frac{1}{i^2}$, that $2Q_i$ are pairwise disjoint and that $\ell(Q_i) \geq \delta$ for all $i\in \N$. Set $u(x) = \sum_{i=1}^\infty \psi_{Q_{2i}}$. Then $u$ is $\frac{3d}{\delta}$-Lipschitz and $|\nabla u| \leq \sum_{i=1}^\infty \frac{3d}{\delta} 1_{Q_{2i}}$.

Since $w$ is doubling, we have $\liminf_{i\to\infty}u_{Q_{2i}}>0$, but $\lim_{i\to\infty} u_{Q_{2i+1}} = 0$. Therefore, the limit does not exist. If $Q_i = Q \times [n_i,n_i+1]$, then $u(\overline{x},t)=1$ whenever $\overline{x}\in Q$ and $t\in [n_{2i},n_{2i}+1]$, and $u(\overline{x},t)=0$ whenever $\overline{x}\in Q$ and $t\in [n_{2i+1},n_{2i+1}+1]$. 
Thus, the limit $\lim_{t\to\infty} u(\overline{x},t)$ does not exist for any $\overline{x}\in Q$.

Next, we show that there exists a weight $w \in \A_p(\mathbb R^d)$ and cubes $Q_i$, with $i\in \N$, with properties $(1)-(3)$. This proves the first claim of the example.  The existence is given by Example \ref{ex:Apweight} with $\beta>0$. For that example, and any sequence 
$Q_i \to \infty,$ we have $\lim_{i\to \infty}w(Q_i)=0$. Hence, for any sequence of cubes of edge lengths bounded away from zero that tends to infinity one can find a Sobolev function for which the averages do not 
converge.

\end{proof}

  The previous examples justify the assumption $\inf_{\ell(Q)=1} w(Q) > 0$. This assumption together with $\Rp(w)<\infty$ and $w\in \A_p(\mathbb R^d)$ suffices for \emph{rough} limits to exist; see Theorem \ref{lisatty}. However, for vertical limits, e.g. Theorem \ref{thm:pekka-mizuta}, we need further assumptions, as the following example indicates. The idea is to place smaller jumps $\psi_{Q_i}$ with diameters going to zero. 

Before presenting the example, we need the following lemma which collects the crucial feature of our construction.

\begin{lemma}\label{lem:muckenhoupt} Suppose that $p\in[1,\infty)$ and that $w\in L^1_{\loc}(\R^d)$ satisfies the following. There is a sequence of cubes $Q_i$ with edge lengths $\ell_i=\ell(Q_i)\leq 1$ so that 
\begin{enumerate}
\item $d(Q_i,Q_j):=\inf \{|x-y| : x\in Q_i, y\in Q_j\} \geq 1$, for distinct $i,j\in \N$;
\item For each $x\not\in \bigcup_{i\in \N} Q_i$, we have $w(x)=1$; 
\item For each $i\in \N$, the weight $w_i$ defined by $w_i =w1_{Q_i} + 1_{\R^d\setminus Q_i}$ belongs to $\A_p(\mathbb R^d),$ with $\A_p$-constant $C$ (independent of $i$). 
\end{enumerate}
Then $w\in \A_p(\mathbb R^d).$
\end{lemma}

\begin{proof}
Recall that the $\A_p$-conditions \eqref{eq:muckenhoupt} and \eqref{eq:muckenhoupt-p-1} involve estimates for balls $B=B(x,r)$. 

First, by Theorem \ref{thm:muckenhouppi}, there exists a constant $D$ so that for every $i\in \N$ the weight $w_i$  is $D$-doubling. If $\hat{Q}_i$ is a cube obtained from $Q_i$ by reflecting it through one of its faces, then
\begin{equation}\label{eq:wibound_1}
 w_i(Q_i) \lesssim w_i(\hat{Q}_i) \leq \ell_i^d.
\end{equation}
By the $\A_p$-condition for each $w_i$, we also have when $p>1$ the estimate
\begin{equation}\label{eq:wibound_2}
\left(\frac{1}{|Q_i|}\int_{Q_i} w_i^{\frac{1}{1-p}}\right)^{p-1} \lesssim  1.
\end{equation}
When $p=1$, we get $w(y) \gtrsim 1$ for a.e. $y\in Q_i$. There are now two cases to consider in verifying the $\A_p$-conditions for $w$.

\begin{enumerate}
\item If $B \cap Q_i \neq \emptyset$ for at most one $i\in \N$, then $w|_B=w_i|_B$, and the $\A_p$-condition follows from that of $w_i$.
\item If $B \cap Q_i \neq \emptyset$ for more than one $i \in \N$, then by the separation condition $r\geq 2^{-1}$. Let $I \subset \N$ be the set of those indices $i$ for which $B \cap Q_i \neq \emptyset$. Since $\ell_i \leq 1$, we have $Q_i \subset 2(1+\sqrt{d}) B$ for each $i\in I$. Thus \eqref{eq:wibound_1} and the properties of $w$ give
\begin{equation}\label{ylaraja}
\frac{1}{|B|} \int_B w dx \leq \frac{1}{|B|} \int_{Q \setminus \bigcup_{i\in \N} Q_i} 1 dx + \frac{1}{|B|}\sum_{i\in I} w(Q_i) \lesssim 1. 
\end{equation}
When $p>1$, we argue similarly, using \eqref{eq:wibound_2} instead of \eqref{eq:wibound_1} to conclude that
\[
\frac{1}{|B|} \int_B w^{\frac{1}{1-p}} dx \leq \frac{1}{|B|} \int_{2B \setminus \bigcup_{i\in \N} Q_i} 1 dx + \frac{1}{|B|}\sum_{i\in I} \int_{Q_i} w^{\frac{1}{1-p}} \lesssim 1. 
\]
The desired inequality follows.
When $p=1$, we have $w(y) \geq 1$ for $y\not\in Q_i$, and $w(y) \geq w_i(y) \gtrsim 1$ when $y\in Q_i.$ In either case, we obtain the $\A_p$-conditions via \eqref{ylaraja}.
\end{enumerate}
\end{proof}

In the following, notice that $\frac{p+d-1}{d}< p$
whenever both $p>1$ and $d\geq 2$ hold.

\begin{example}\label{ex:mainexample} For all $q,p\in (1,d)$ with $\frac{p+d-1}{d} < q\leq p$, there exists  $w\in \A_q(\mathbb R^d)$ which satisfies \[\inf_{\ell(Q)=1} w(Q) > 0,\] which has $\Rp(w)<\infty$ and which satisfies the following. There exists a function $u\in \dot W^{1,p}(\R^d, w)$ so that for a.e. $\overline{x}\in \R^{d-1}$ the limit $\lim_{t\to\infty} u(\overline{x},t)$ fails to exist.  
\end{example}

\begin{proof} First, we construct a sequence of ``small'' cubes $\{\hat{Q}_i\}_{i\in \N}$ in $\R^{d}$ with $\ell(\hat{Q}_i) \leq 1/2$, of pairwise distance at least $2$, with  $\lim_{i\to\infty} \ell(\hat{Q}_i)=0$, $\hat{Q}_i \to \infty$, and so that their projections cover a.e point of $\R^{d-1}$ infinitely often\footnote{With some more work, one could also construct a sequence of cubes $Q_i$ so that \emph{every} $\overline{x}$ would be covered by the projections of $Q_i$ for infinitely many $i\in \N$.}. Then, we construct a weight $w \in \A_p(\mathbb R^d)$ with $\Rp(w)<\infty$ and 
\begin{equation}\label{eq:sum_wbound}
\sum_{i=1}^\infty w(2\hat{Q}_i)\ell(\hat{Q}_i)^{-p} < \infty.
\end{equation}
 The function $u=\sum_{i=1}^\infty \psi_{2\hat{Q}_i}$ will then serve as the desired counter-example.

Consider the Gaussian probability measure $P$ on $\R^{d-1}$ given by $dP=\frac{e^{-\frac{|x|^2}{2}}}{(2\pi)^{\frac{d-1}{2}}}dx$. Let $\ell_n = \frac{1}{2n^{\frac{1}{d-1}}}$.  Choose a random sequence $\{\overline{x_i}\}_{i\in \N}$ so that each $\overline{x_i} \in \R^{d-1}$ is chosen independently and with distribution $P$. Define $\hat{x}_i = (\overline{x_i}, 4i) \in \R^d$. A straightforward calculation using Borel-Cantelli shows that, almost surely, the cubes $\hat{Q}_i=Q(\hat{x}_i, \ell_i)$ satisfy the property that a.e. $\overline{x} \in \R^{d-1}$ is covered by infinitely many of the projections  onto $\R^{d-1}$ of $\hat{Q}_i$.

Next, let \[ w(x) = \min\{ 1,\inf_{i\in \N} |x-\hat{x}_i|^\alpha \}\] and fix $\alpha \in (p-1,d(q-1))$, which is possible since $\frac{p-1}{d}<q-1$. Since for any $x\in \R^d$ the only terms contributing to the infimum come from those $\hat{x}_i$ that are contained in $B(x,1)$, it is straightforward to show that the latter infimum is actually a minimum. 

First, that $w\in \A_q(\mathbb R^d)$ follows from Lemma \ref{lem:muckenhoupt} using the cubes $Q_i = Q(\hat{x}_i,1)$ since $\alpha \in [0,d(q-1))$. In this case, a fairly direct and classical calculation shows that $w_i(x) = \min(1, |x-\hat{x}_i|^\alpha)$ is an $\A_p$-weight with constant independent of $i\in\N$. 
The $\A_p$-conditions for $w_i$ can be verified via a case study involving integration over polar coordinates. 

Also, for $A_k=\{x\in\mathbb R^d : 2^{-k} <|x|\leq 2^{k+1}\}$, we have $w(A_k) \gtrsim 2^{kd}$.
%, since $w(x)=1$ whenever $x=(\overline{a},t)$ with $|\overline{a}|\geq 1$. 
Thus, the requirement $\Rp(w) <\infty$ follows from the definition.

The condition \[\inf_{\ell(Q)=1} w(Q) > 0\] follows from the observation that each cube $Q$ with $\ell(Q)=1$, we have $w \geq 8^{-\alpha}$ for at least half of the volume of $Q$, since $Q$ can intersect at most one ball $B(\hat{x}_i, 1/8)$. This ball can cover at most a half of the volume, and outside  it $w\geq 8^{-\alpha}$.

Finally, we verify \eqref{eq:sum_wbound}:

\begin{equation}
\sum_{i=1}^\infty w(2\hat{Q}_i)\ell(\hat{Q}_i)^{-p} \lesssim \sum_{i=1}^\infty \ell(\hat{Q}_i)^{\alpha + d}\ell(\hat{Q}_i)^{-p} \lesssim \sum_{i=1}^\infty \frac{1}{i^{\frac{d+\alpha-p}{d-1}}} < \infty,
\end{equation}
since $\alpha +d-p > d-1$ holds whenever $\alpha>p-1$. 
\end{proof}

We close this section with an example of a product weight $w_{P}$, and of a radial weight $w_R$, for which radial limits exist but no vertical limits exist.

\begin{example} Suppose that $p \in [1,d)$. Let $w_P(x,y)=\min(1,y^{-\alpha})$, $\alpha \in (0,\min(1,d-p))$, and $w_R(x)=\min(1,|x|^{-\alpha})$ with $\alpha \in (0,d-p)$. Then, $\Rp(w)<\infty$ for $w\in \{ w_P, w_R\}$,  and there exists a $u \in \dot W^{1,p}(\R^d,w)$ so that for all $\overline{x} \in \R^{d-1}$ the limit $\lim_{t\to\infty} u_(\overline{x},t)$ fails to exist.
\end{example}

\begin{proof} The weights $w_P$ and  $w_R$ are in $\A_p(\mathbb R^d)$ for all $\alpha \in (0,1)$ and $\alpha\in (0,d)$, respectively; see e.g. the proof of \cite[Theorem 1.1.]{BILTV}. Next, $w_P(A_k) \sim w_R(A_k) \sim 2^{k(d-\alpha)}$, where $A_k=\{x\in \R^d : 2^k \leq |x|\leq 2^{k+1}\}$. Thus, by definition,  whenever $d-\alpha>p$, we have $\Rp(w)<\infty$.

Choose $\beta \in (0,\min(\alpha/(d+p),1))$. 
Let $Q_i = Q((\overline{0},2^i), 2^{\beta i})$ for $i\geq 2$, and let $u = \sum_{i=2}^\infty \psi_{2Q_i}$. Then $\int_{\mathbb R^d} |\nabla u|^p wdx \lesssim \sum_{i=1}^\infty 2^{\beta (d+p) i}2^{-\alpha i}$. Since every $\overline{x} \in \R^{d-1}$ belongs to all but finitely many of the projections of $Q_i$, we have that no vertical limit exists for $u$. 
\end{proof}

\section{Vertical Limits}\label{sec:vertical-limit}

In this section, we discuss the case for vertical limits. We will divide this into four parts: first rough averages, then pointwise limits and finally the cases of product and radial weights.
\subsection{Rough averages}

Before embarking on the proof we record a conclusion regarding rough average limits.

\begin{lemma}\label{lem:w-average} Let 
%$C \geq 1$.  
$C>2$ and $p\in[1,\infty)$.
Let $Q_i=Q(x_i,\ell_i)$ be a sequence of cubes with $Q_i \to \infty$ and so that 
%$|x_i|/C \leq \ell_i \leq |x_i|$.
$\sqrt{d}\ell_i/2\leq |x_i|\leq C\ell_i$.
 If $\Rp(w)<\infty$, then for all $u\in \dot W^{1,p}(\R^d,w)$ we have

$$\lim_{i\to\infty} u_{Q_i}=c$$
where $c$ is the unique almost sure radial limit of $u$.
\end{lemma}

\begin{proof}There exist constants $C_1>0,C_2>0$ independent of $i\in\mathbb N$ such that $Q_i\subset \tilde{A_i}$ where $\tilde{A_i}:=\{ x\in\mathbb R^d: C_1 \ell_i\leq |x|\leq C_2\ell_i\}$. 
%Let $A_i = \{x\in \R^d : \ell_i/2 \leq |x|\leq 2C\ell_i\}$. 
Then there is a constant $C>0$ such that 

$$|u_{\tilde{A_i}}-u_{Q_i}| \lesssim C\ell_i \left(\avint_{\tilde{A_i}} |\nabla u|^p d\mu \right)^{\frac{1}{p}}= C \frac{\ell_i}{w(\tilde{A_i})^{\frac{1}{p}}}  \left(\int_{\tilde{A_i}} |\nabla u|^p d\mu\right)^{\frac{1}{p}}.$$ 
Here we use the uniform $p$-Poincar\'e inequality as in \cite{sobolevmet} for the John domains $\tilde{A_i}$.

Note that $\Rp(w) <\infty$, $\tilde{A_i} \subset \{x\in \R^d : |x| \geq C_1\ell_i \}$ and $\lim_{i \to \infty} \ell_i = \infty$. {Then, $\lim_{i\to\infty}l_i/w^{\frac{1}{p}}(\tilde{A_i})=0$ by a similar argument as in the proof of Lemma \ref{lem3.1}. This, together with the fact that $|\nabla u| \in L^p(\R^d,w)$, shows that the right-hand side  converges to $0$ when $i\to \infty$.} That is, $\lim_{i\to\infty} |u_{\tilde{A_i}}-u_{Q_i}| = 0$. It thus suffices to prove that $\lim_{i\to \infty} u_{\tilde{A_i}}=c$.

Fix next an $\epsilon>0$. By Lemma \ref{lem3.1}, for a.e. $\xi \in \mathbb S^{d-1}$, we have $\lim_{r\to\infty} u(r\xi) = c$. Then, by Egorov, there exists a set $F \subset \mathbb S^{d-1}$ with $\mathcal{H}^{d-1}(F)\geq \frac{\mathcal{H}^{d-1}(\mathbb S^{d-1})}{2}$ and an $i_0$ so that for all  $i\geq i_0$ and for all $r\in [\ell_{i}/2,2C\ell_{i}]$ and all $\xi \in F$ we have $|u(\xi r)-c|\leq \epsilon.$ 

Define a sequence of sets by $E_{i}=\{r\xi : r\in [\ell_{i}/2,2C\ell_{i}], \xi \in F\}$ and notice that  $E_{i_0}\subset\tilde{A}_{i_0}$ by construction. Then, $\mathcal{H}^{d}(E_{i}) \geq 2^{-1}\mathcal{H}^{d}(\tilde{A_i})$. Since $w\in \A_p(\mathbb R^d)$, H\"older's inequality implies that there exists a constant $\delta>0$ so that $w(E_i)\geq \delta w(\tilde{A_i})$ for all $i\in\N$ with $i\geq i_0$; see \cite[Chapter V]{Stein} for details. 

In particular, the $p$-Poincar\'e inequality implies that 

$$|u_{E_{i}} - u_{\tilde{A_i}}|\lesssim_\delta C\ell_i \left(\avint_{\tilde{A_i}} |\nabla u|^p d\mu \right)^{\frac{1}{p}}= C \frac{\ell_i}{w(\tilde{A_i})^{\frac{1}{p}}} \left(\int_{\tilde{A_i}} |\nabla u|^p d\mu\right)^{\frac{1}{p}}.$$
Again, we obtain that $\lim_{i\to\infty} |u_{E_i}-u_{\tilde{A_i}}| = 0$. However, $|u_{E_i}-c|\leq \epsilon$ for all $i\geq i_0$. Thus, $$\limsup_{i\to\infty} |u_{\tilde{A_i}}-c| \leq \epsilon.$$ Since $\epsilon>0$ is arbitrary, the claim follows.

\end{proof}

\begin{remark}\label{rmk:Rpcomparison}
Let $a\in \R^d$, and define the translated weight $w^a(y) \defeq w(y-a)$. The quantity $\Rp(w)$ is not translation invariant, and thus $\Rp(w)$ may be different from $\Rp(w^a)$. However, these quantities are comparable, since $w$ is a doubling weight by Theorem \ref{thm:muckenhouppi}. Indeed, let $A_i = \{x\in \R^d : 2^i \leq |x| \leq 2^{i+1}\}$ and $A_i^a = \{x \in \R^d : 2^i \leq |x-a| \leq 2^{i+1}\}$. Let $i_0 \in \N$ be chosen so that $|a|\leq 2^{i_0}$ and $i_0 \geq 1$. Then, for each $i\in\N$ we have $w(A_i^a) \gtrsim c_D^{i_0+5} w(B(0, 2^{i+1})) \geq c_D^{i_0+5} w(A_i)$. From this, and the definition of $\Rp(w)$ we get a constant $C_{i_0}$ so that
\[
\Rp(w^a) \leq C_{i_0} \Rp(w), \text{ for all } a \in B(0,2^{i_0}).
\]
\end{remark}

For the following lemma, we introduce the notion of a half-space. Given $t>0$, define the half-space $H_t=\{(\overline{x},t) \in \R^d : t >0\}$.

\begin{lemma}\label{lem:av-est} Assume $p\in[1,\infty)$. Let $w\in \A_p(\mathbb R^d)$ with $\Rp(w)<\infty.$ Suppose that $u \in \dot W^{1,p}(\R^d, w)$. Let $Q\subset \R^{d-1}$ be a cube centred at $\overline{x}\in\mathbb R^{d-1}.$ Given $t>0$, set $\tilde{Q}=Q \times [t,t+\ell(Q)].$
If $p>1,$ then

\[
\avint_{\tilde{Q}}|u-c|d\mu \lesssim_{\ell(Q)} \ell(Q)\left(\int_{H_{t}} |\nabla u|^p wdx\right)^{\frac{1}{p}}  (\Rp(w^{(\overline{x},t)}))^{\frac{p-1}{p}},
\]
where $c$ is the unique almost sure radial limit. When $p=1$, the same bound  holds with $\mathcal R_1(w)$ replacing $(\Rp(w^{(\overline{x},t)}))^{\frac{p-1}{p}}$.
\end{lemma}

\begin{proof}
Let $Q_0 = \tilde{Q}$ and $Q_i = 2^iQ \times [t+2^i\ell(Q), t+2^{i+1} \ell(Q)]$ for $i \geq 1$. Let $\hat{Q}_1 = 2Q \times  [t,t+2\ell(Q)]$ and let $\hat{Q}_i = 2^iQ \times [t+2^{i-1}\ell(Q), t+2^{i+1} \ell(Q)]$ for $i \geq 1$. By construction, $Q_i,Q_{i+1} \subset \hat{Q}_{i+1}$ for all $i\in \N$ and $\hat{Q}_i \subset H_{t}$. Also, $\hat{Q}_i \cap \hat{Q}_j=\emptyset$ when $|i-j|\geq 3$.

By the $p$-Poincar\'e inequality and doubling, we have
\[
|u_{Q_{i+1}}-u_{Q_i}| \lesssim \ell(\hat{Q}_{i+1} ) \left(\avint_{\hat{Q}_{i+1} }|\nabla u|^p d\mu\right)^{\frac{1}{p}}\lesssim 2^i \ell(Q) \left(\avint_{\hat{Q}_{i+1}} |\nabla u|^p d\mu\right)^{\frac{1}{p}}.
\]

Recall that $\lim_{i \to \infty} u_{Q_i}=c$, by Lemma \ref{lem:w-average}. Thus, summing the previous estimate over $i$ and using H\"older's inequality gives

\begin{align*}
\avint_{\tilde{Q}}|u-c|d\mu \leq \avint_{\tilde{Q}}|u-u_{Q_0}|d\mu + \sum_{i=1}^\infty |u_{Q_{i+1}}-u_{Q_i}| &\lesssim \sum_{i=1}^\infty  2^i \ell(Q) \left(\avint_{\hat{Q}_i} |\nabla u|^p wdx\right)^{\frac{1}{p}} \\
&\leq \ell(Q) \left(\sum_{i=1}^\infty  2^{\frac{ip}{p-1}} w(\tilde{Q}_i)^{\frac{-1}{p-1}} \right)^{\frac{p-1}{p}}  \left(\sum_{i=1}^\infty \int_{\hat{Q}_i} |\nabla u|^p d\mu \right)^{\frac{1}{p}}.
\end{align*}

The case of $p=1$ is similar, but uses the fact that $\sup_{i\in \N}\ell(\hat{Q}_i)w(\hat{Q}_i)^{-1} \lesssim \mathcal R_1(w^{(\overline{x},t)})$.

\end{proof}

\begin{proof}[Proof of Theorem \ref{thm:cubed-average}] Suppose that every $u\in \dot W^{1,p}(\R^d,w)$ has a unique almost sure vertical limit. If $\Rp(w)=\infty$, then Lemma \ref{lem3.5} yields a function $u\in  W^{1,p}(\R^d,w)$ so that $\lim_{|x|\to\infty} u(x)=\infty$. This is a contradiction, and thus $\Rp(w)<\infty$. Suppose that $\inf_{z\in \N} w(Q \times [z,z+1])=0.$ Then there exists an increasing sequence $(n_i)_{i\in\N}$ with $n_i \to\infty$ so that $\lim_{i\to\infty} w(Q\times [n_i,n_i+1])=0$.  Let $Q_i = Q\times [n_i,n_i+1]$. Then,  Example \ref{example:rough} applied to the cubes $Q_i,$ gives a function $u\in W^{1,p}(\R^d,w)$ so that $\lim_{t\to\infty} u(\overline{x},t)$ does not exist for any $\overline{x}\in Q$. Consequently, we must have $\inf_{z\in \N} w(Q \times [z,z+1])>0$.
\end{proof}

\begin{proof}[Proof of Theorem \ref{lisatty}] 

We begin by verifying that $(2)$ implies \eqref{eq:lower-mass-bound}.  
Suppose that \eqref{eq:lower-mass-bound} does not hold. Then, there is a sequence $Q_i\to\infty$ so that $\lim_{i\to\infty} w(Q_i) = 0$ but $\ell(Q_i)=1$.  
By passing to a subsequence, we may assume that $Q_i$ are pairwise disjoint and that $\sum_{i\in \N}w^{1/(2p)}(Q_i)<\infty.$ Let $u=\sum_{i\in \N} \frac{1}{w(Q_i)^{1/(2p)}}\psi_i$, where $\psi_i$ is a $2$-Lipschitz function with $\psi_i|_{\frac{1}{2}Q_i}=1$ and $\psi_i|_{\R^d \setminus Q_i}=0$. 
Then $u\in W^{1,p}(\R^d,w) \subset \Wdp(\R^d,w)$. 
However, $\avint_{Q_i} u d\mu \geq \frac{1}{w(Q_i)^{1/(2p)}}\to \infty$ as $i\to\infty$, which contradicts $(2)$.

Next, we assume that \eqref{eq:lower-mass-bound} holds. Let $\hat Q$ be a cube of edge length $R\ge 1.$ By considering the maximal collection of pairwise disjoint subcubes of edge length one of $\hat Q,$ we
conclude from \eqref{eq:lower-mass-bound} that $w(\hat Q)\gtrsim R^d.$
It especially follows that $w(\{x\in \R^d:\  2^i\le |x|<2^{i+1}\})\gtrsim 2^{id}$ and hence that $\Rp(w)<\infty$ since $p<d.$ 

Fix an  $R>0$ and a  cube $Q$  
so that $\infty>\ell(Q)\ge \delta>0$ and  $Q\cap B(0,R) = \emptyset$ for some constant $\delta$ independent of $Q$. Let $Q_0 = Q$. Form a sequence of cubes $Q_n$ recursively by defining $Q_{n+1}$ to be the cube centred at a corner $v_n$ of $Q_n$ which is furthest away of the origin, and with twice the edge length. This gives a sequence $Q_n$ with $\ell(Q_n) = 2^n \ell(Q)$ and so that $Q_n \cap B(0,R)= \emptyset$, $Q_0=Q$ and  $Q_k \subset Q_l$ for $k<l$.  
Let  $u \in \Wdp(\R^d,w)$. 

Let $\tilde{Q}_n \subset Q_n$ be the orthant of $Q_n$ whose center is furthest away of the origin. Denote the center by $x_{\tilde{Q}_n}$. 
We have $\ell(\tilde{Q}_n)=\ell(Q_n)/2$ and there exists a constant $C=C(R,\ell(Q))$ so that Lemma \ref{lem:w-average} is satisfied, that is $\sqrt{d}/2\ell(\tilde{Q}_{n})\leq |x_{\tilde{Q}_{n}}|\leq C \ell(\tilde{Q}_{n})$ for all integers $n \geq 0$.

Notice  that $\mathcal R_p(w)<\infty$ by Remark \ref{remark1.14}.
Then, by Lemma \ref{lem:w-average} we get $\lim_{n\to\infty} u_{\tilde{Q}_{n}}=c$, where $c$ is the almost sure radial limit.

Then, we have from the $p$-Poincar\'e inequality that

\[
|u_{Q_n}-u_{\tilde{Q}_{n}}|
%\lesssim 2^n \ell(Q) (|\nabla u|^p)_{Q_{n}}^{1/p}
\lesssim  \frac{2^n\ell(Q)}{w(Q_{n+1})^{\frac{1}{p}}} \left( \int_{\mathbb R^d\setminus B(0,R)} |\nabla u|^p(x) w(x)dx \right)^{1/p}
\]
and
\[
|u_{Q_n}-u_{Q_{n+1}}|
%\lesssim 2^n \ell(Q) (|\nabla u|^p)_{Q_{n+1}}^{1/p}
\lesssim  \frac{2^n\ell(Q)}{w(Q_{n+1})^{\frac{1}{p}}} \left( \int_{\mathbb R^d\setminus B(0,R)} |\nabla u|^p(x) w(x)dx \right)^{1/p}.
\]

We have {$w(Q_{n+1}) \gtrsim (2^n \ell(Q))^d$}.  By recalling that $p<d$, we get from the first bound that $\lim_{n\to \infty}|u_{Q_n}-u_{\tilde{Q}_n}|=0$. Therefore $\lim_{n\to\infty} u_{Q_n}=c$.

Next, summing the second bound over $n$  we get
\[
\sum_{n\in\mathbb N} |u_{Q_n}-u_{Q_{n+1}}| \lesssim _\delta \left( \int_{\mathbb R^d\setminus B(0,R)} |\nabla u|^p(x) w(x)dx \right)^{1/p}<\infty.\]
 
Thus, by a telescoping sum, we get

$$|u_{Q}-c|\leq \sum_{n\in\mathbb N} |u_{Q_n}-u_{Q_{n+1}}| \lesssim_\delta \left( \int_{\mathbb R^d\setminus B(0,R)} |\nabla u|^p(x) w(x)dx \right)^{1/p}.$$

Now, if $R \to \infty$, the right-hand side converges to zero. Thus, if $Q_i \to \infty$, for every $R>0$, we can find a $N$ so that $Q_i \cap B(0,R) = \emptyset$ for all $i \geq N$. This gives $\lim_{R\to\infty} u_{Q}=c.$ The proof is complete.

\end{proof}

\subsection{Pointwise limits}

First, we need an auxiliary result. This is a stronger form of the necessary condition in Theorem \ref{thm:cubed-average}.

\begin{lemma}\label{Syl} If $w\in \A_p(\mathbb R^d)$ and $\sup_{t>0}\Rp(w_t)=\infty$, then there is a function $u\in \Wdp(\R^d,w)$ so that for no $\overline x\in B(\bar 0,1)\subset \R^{d-1}$ does the limit $\lim_{t\to\infty} u(\overline x,t)$ exist.
\end{lemma}

\begin{proof}
If $\Rp(w_t)=\infty$ for some $t>0$, then the claim follows from Lemma \ref{lem3.5}. Thus, assume that $\Rp(w_t)<\infty$ for each $t>0$. 

%It suffices to show that there is a function $u\in \dot W^{1,p}(\mathbb R^d,w)$ such that $\lim_{|x|\to\infty}u(x)=\infty$. 
Let $A_i(t):=B(O(t),2^{i+1})\setminus B(O(t),2^i)$ be the translated annulus $A_i=B(0,2^{i+1})\setminus B(0,2^i)$ for the center $O(t)=(\overline{0},t)$. Then, $w_t(A_i)=w(A_i(t))$. We have 
\[
\sup_{t>0}\Rp(w_t)=\sup_{t>0}\sum_{i\in\mathbb N}(2^{i})^{\frac{p}{p-1}}w^{\frac{1}{1-p}}(A_i(t)) \text{\rm \ \ if $p>1$, and \ \ }\sup_{t>0}\mathcal R_1(w_t)=\sup_{t>0}\sup_{i\in \N} \left(2^{i}w^{-1}(A_i(t))\right).
\]
Since $\mathcal R_{p}(w_t)<\infty$ for each $t$ and $\sup_{t>0}\mathcal R_p(w_t)=\infty$, there is a sequence $\{t_k\}_{k\in\mathbb N}$ with $t_{k+1}>t_k+1$ such that $\mathcal R_p(w_{t_k})>2^k$. Let $L_k\defeq t_k-t_{k-1}$.  By Theorem \ref{doubling-1811}, $w$ is doubling, and we can show that there exists a constant $C>0$ so that $w(A_i(t_{k})) \geq C w(A_i(t_{k-1}))$ when $2^i\geq L_k/4$. By passing to a subsequence, we can assume that $\mathcal R_p(w_{t_{k}}) \geq 2C^{\frac{1}{1-p}} \mathcal \mathcal R_p(w_{t_{k-1}})$ when $p>1$, or $\mathcal R_1(w_{t_{k}}) \geq 2C^{-1} \mathcal \mathcal R_1(w_{t_{k-1}})$ for $p=1$, for all integers $k\geq 2$. Then,
\[\sum_{i\in \N, 2^i \geq L_k/4} (2^i)^{\frac{p}{p-1}}w^{\frac{1}{1-p}}(A_i(t_k)) \leq C^{\frac{1}{1-p}} \mathcal R_p(w_{t_{k-1}}) \leq \frac{1}{2}\mathcal R_p(w_{t_{k}}) \text{\rm \ \ if $p>1$,}\]
and
\[\sup_{i\in \N, 2^i \geq L_k/4} 2^{i}w^{-1}(A_i(t)) \leq C^{-1} \mathcal R_1(w_{t_{k-1}}) \leq \frac{1}{2}\mathcal R_1(w_{t_{k}}) \text{\rm \ \ if $p=1$.}\]
Therefore,
\begin{align*}
&\sum_{i\in\mathbb N, 2^i<L_k/4}(2^i)^{\frac{p}{p-1}}w^{\frac{1}{1-p}}(A_i(t_k))\geq\frac{1}{2} \mathcal R_p(w_{t_k})>2^{k-1} \text{\rm \ \ if $p>1$,} \\
&\sup_{i\in\mathbb N, 2^i<L_k/4} 2^iw^{-1}(A_i(t_k))\geq\frac{1}{2} \mathcal R_1(w_{t_k})>2^{k-1} \text{\rm \ \  if $p=1$.}
\end{align*}
For $p>1$, define annuli $A_k$ by $A_k\defeq \bigcup_{i\in\mathbb N, 2^i<L_k/4}A_i(t_k)$. If $p=1$, let $i_k$ be such that $2^{i_k}<L_k$ and $2^{i_k}w^{-1}(A_{i_k}(t_k))>2^{k-1}$. In this case, set $A_k:=A_{i_k}(t_k)$. Next, we proceed in a similar way as in the
proof of Lemma \ref{lem3.5} and set

	\begin{equation}\notag
	\label{gkdef}  g_k(x)= \sum_{i\in \N, 2^i<L_k/4}\frac{(2^i)^{\frac{1}{p-1}}w^{\frac{1}{1-p}}(A_{i}(t_k))}{\sum_{i\in \N, 2^i<L_k/4} (2^i)^{\frac{p}{p-1}}w^{\frac{1}{1-p}}(A_{i}) } \chi_{A_{i}(t_k)}(x) \text{\rm \  if $p>1$, and\ }
	g_k(x)=2^{-i_k}\chi_{A_k}(x) \text{\rm \ when  $p=1$.}
	\end{equation}

	By passing to a subsequence, we can ensure that $(t_{k}-t_{k-1})/2>t_{k-1}+2L_{k-1}$ for all $k\in \N$, which guarantees that the sets $A_k$ are disjoint, and that the point $O=(\overline{0},t_1-4L_1)$ is not contained in any of the balls $B(t_k,2L_k)$, for $k\in \N$.  Define $g = \sum_{k=2}^\infty g_{2k+1}$. Then, $\int_{\mathbb R^d} g^p d\mu \leq \sum_{k=2}^\infty\int_{\mathbb R^d} g_{2k+1}^p d\mu< \infty$, since the supports are disjoint.  Define $u(x) = \inf_{\gamma} \int_{\gamma}gds$
	%$ \min(\int_\gamma \sum_{k=1}^\infty g_{2k+1},1)$, 
	 where $\gamma$ is any rectifiable curve connecting $O$ to $x$. As in the proof of Lemma \ref{lem3.5}, we have $|\nabla u_k(x)|\leq g(x)$ for $\mu$-a.e. $x\in \R^d$. Since $u_k\leq 1$, we have $u_k \in \dot W^{1,p}(\R^d,w)$.
	
	Next, let $O_k=(\overline{0},t_k)$. Any curve which connects $B(O_k,1)$ to $O$ must pass through the annulus $A_k$. On the other hand, if $k$ is odd, there is a rectifiable curve connecting $B(O_k,1)$ to $O$ which does not pass through any $A_l$ for $l$ even. Thus 
\[
\begin{cases}
	u=1  \text{\rm \ \ on\ \ }B(O_{2k+1},1),\\
	u=0  \text{\rm \ \ on\ \ }B(O_{2k},1).
\end{cases}
\]
Then $\lim_{t\to\infty}u(\bar x,t)$ does not exist when $\bar x\in B(\bar 0,1)$ and $u\in \dot W^{1,p}(\mathbb R^d,w)$.
\end{proof}

\begin{proof}[Proof of Theorem \ref{thm:vertical-lowerbound}]

First, we prove that $(1)\Longrightarrow (2)$. Let $Q\subset\mathbb R^{d-1}$ be a cube of unit size. 
It suffices to show that
\begin{equation}\label{eq:5.2-1109}
\lim_{t\to\infty, (\bar x,t)\in Q\times [n,n+1]}|u(\bar x,t)-u_{Q\times[n,n+1]}|=0 \text{\rm \ \ for $\mathcal H^{d-1}$-a.e $\bar x\in Q$. }
\end{equation}
 Let $E_n=\{ (\bar x,t)\in Q\times[n,n+1]: |u(\bar x,t)- u_{Q\times[n,n+1]}|>a_n \}$ where $\{a_n\}_{n\in\mathbb N}$ satisfies 
 \[
 a_n>0, \lim_{n\to\infty,n\in\mathbb N}a_n=0, \sum_{n\in\mathbb N}\frac{1}{a_n^p}\int_{Q\times[n,n+1]}|\nabla u|^pd\mu<\infty.
 \]
 Here the sequence $\{a_n\}_{n\in\mathbb N}$ with above properties exists see for instance \cite[Lemma 3.3]{KN21}.
 For every $x\in E_n$ which is a Lebesgue point (with respect to $\mu$) of $u$, we have that $a_n\lesssim  \mathcal M^{1/p}_{p-\alpha,\text{diam} (E_n)}|\nabla u|^p(x)$ for any $0<\alpha<p<{qd}$ by the $p$-Poincar\'e inequality. Let $LE_n$ be a set of all Lebesgue points (with respect to $\mu$) in $E_n$. Notice that $\inf_{n\in\mathbb N}w(Q\times[n,n+1])>0$. By Theorem  \ref{thm2.8} applied to the zero extension of $|\nabla u|$ to $Q\times[n,n+1]$, we obtain that 
\[
\mathcal H^{{qd}-p+\alpha}_\infty(LE_n)\leq \mathcal H^{{qd}-p+\alpha}_\infty(\{x\in Q\times[n,n+1]: \mathcal M^{1/p}_{p-\alpha,\text{diam}(E_n)}u(x)\gtrsim a_n \})\lesssim \frac{1}{a_n^p}\int_{Q\times[n,n+1]}|\nabla u|^pd\mu.
\] Let $A^*$ is the projection of $A\subset\mathbb R^d$ into $\mathbb R^{d-1}$. Hence $\mathcal H^{{qd}-p+\alpha}_\infty(LE)=0$ where $LE=\bigcap_{m\in \mathbb N} \bigcup_{n\geq m} (LE_n)^*$ and so $\mathcal H^{{qd}-p+\alpha}(LE)=0$. Notice that there is $0<\alpha<p$ such that ${qd}-p+\alpha\leq d-1$ because  ${qd}-(d-1)<p$. It follows that $\mathcal H^{d-1}(LE)=0$. 

By \cite[Theorem 4.4]{HKi98}, we have $\mathcal H^{qd-p+\alpha}({\rm NL}_\mu(u))=0$, where ${\rm NL}_\mu(u)$ is the set of non-Lebesgue points of $u$ with respect to the weighted measure $\mu$. Since ${qd}-(d-1)<p$, $\mathcal H^{d-1}(\bigcup_{n\in \mathbb N}E_n\setminus LE_n)=0$. Therefore, $\mathcal H^{d-1}(\bigcap_{m\in\mathbb N}\bigcup_{n\geq |m|}E_n^*)=0$, and hence  \eqref{eq:5.2-1109} follows.

Next, we show that $(2) \Longrightarrow (1)$. Assume that every $u\in \dot W^{1,p}(\R^d,w)$ has a unique almost sure vertical limit. Then, by Lemma \ref{Syl}, we obtain $\sup_{t>0} \Rp(w_t)<\infty$. Fix a cube $Q \subset \R^{d-1}$ centered at $\overline{x}$ and consider $Q(t)=Q\times [t+\ell(Q)]$. Then, by Lemma \ref{lem:av-est}, we  have
\[
\avint_{Q(t)} |u-c| d\mu \lesssim_{\ell(Q)} \ell(Q)\left( \int_{H_{t}} |\nabla u|^p d\mu \right)^{\frac{1}{p}} (\Rp(w^{(\overline{x},t)})^{\frac{p-1}{p}}
\]
if $p>1$ and the same bounds holds with $\mathcal R_1(w)$ replacing $(\Rp(w^{(\overline{x},t)})^{\frac{p-1}{p}}$ if $p=1$,
where $c$ is the almost sure radial limit.
We have $w^{(\overline{x},t)}=w_t^{(\overline{}{x},0)}$. Thus Remark \ref{rmk:Rpcomparison} together with $\sup_{t>0} \Rp(w_t)<\infty$ gives $\sup_{t>0} \Rp(w^{(\overline{x},t)})<\infty$. Thus, sending $t\to\infty$, we get
$\lim_{t\to\infty} u_{Q(t)}=c,$ as claimed.
\end{proof}

\begin{proof}[Proof of Theorem \ref{thm:pekka-mizuta}]

Consider an arbitrary point $a\in \R^d$. By $\inf_{\ell(Q)=1} w(Q)>0$, $p<d$ and the argument in the second paragraph of the proof of Theorem \ref{lisatty}, we  obtain $\Rp(w^{a})<C<\infty$, with the bound independent of $a$.
 
 Let $\tilde Q\subset\mathbb R^{d-1}$ be a cube.
 Let $Q(T) = \tilde{Q} \times [T,T+\ell(\tilde{Q})]$ where $T>0$. 
 It then follows from Lemma \ref{lem:av-est}, that $\lim_{T \to \infty} u_{Q(T)}=c$, where $c$ is the almost sure radial limit. 
It suffices to show that $\lim_{t\to\infty}|u(\overline x,t)-Q(n)|=0$ for $\mathcal H^{d-1}$-a.e $\overline x\in Q$ where $(\overline x,t)\in Q(n)$, $n\in\mathbb N$.
 
 Let $t = \frac{1}{q-1}$. We have by the definition of $\A_q(\R^d)$ that there is a constant $L$ so that
 
 \[ \avint_{Q(T)}wdx \left( \avint_{Q(T)}w^{-t}dx\right)^{1/t} \leq L.
 \]
 Since $\inf_{\ell(Q)=1} w(Q)>0$, and since $w$ is a doubling weight by Theorem \ref{thm:muckenhouppi}, we get  that $\inf_{T\in \R} w(Q(T))< \infty$. Thus,  by combining the previous two claims, there is a constant $M$ so that for all $T>0$ we have
 \[
 \int_{Q(T)}w^{-t}dx \leq M.
 \]

 Since $x\to \frac{x}{x+1}$ is increasing we have from $t<\frac{d}{p-1}$ that $\frac{t}{1+t}p > \frac{d}{d-1+p}p$. Thus, we can choose  $p'<p$ and $\epsilon>0$ so that  
 $\frac{dp}{d-1+p} <1+\epsilon < p'<\frac{t}{1+t}p$. A direct calculation shows {$s\defeq \frac{p'}{p}\frac{d-1}{d-1-\epsilon}> \frac{p'}{p}\frac{d-1}{d-\frac{dp}{d-1+p}}=p'\frac{d-1+p}{dp}>1$.} 
 
Notice that  $p'>1+\epsilon$ and $p>p'$. Let {$\tau = t(p-p')/p'> 1$} and let $\tau^*$ be the H\"older conjugate of $\tau$, $1/\tau+1/\tau^*=1$. Using H\"older's inequality together with $s\geq 1$ we get

\begin{align*}
\sum_{n=1}^\infty \left(\int_{Q(n\ell(\tilde{Q}))} |\nabla u|^{p'} dx\right)^{\frac{d-1}{d-1-\epsilon}} &\leq \sum_{n=1}^\infty \left(\int_{Q(n\ell(\tilde{Q}))} |\nabla u|^p w dx\right)^{s}\left(\int_{Q(n\ell(\tilde{Q}))} w^{-\frac{p'}{p-p'}} dx\right)^{\frac{(p-p')(d-1)}{p(d-1-\epsilon)}}  \\
 &\leq M^\frac{(p-p')(d-1)}{\tau p(d-1-\epsilon)} |Q(n)|^{\frac{(p-p')(d-1)}{\tau^* p(d-1-\epsilon)}} \left(\sum_{n=1}^\infty \int_{Q(n\ell(\tilde{Q}))} |\nabla u|^p w dx \right)^{s} \\
  &\lesssim \left( \int_{\tilde{Q}\times[1,\infty)} |\nabla u|^p w dx \right)^{s}<\infty.
\end{align*}

We define the sets 
$F_{n,\delta}=\{(\overline x,t) \in Q(n) : |u(\overline x,t)-u_{Q(n)}|\geq \delta\}$.  
From the definition of Hausdorff content, we have the elementary bound $\cH^{d-1}_\infty(F_{n,\delta/2})\leq  \cH^{d-1+\epsilon}_\infty(F_{n,\delta/2})^{\frac{d-1}{d-1+\epsilon}}$. This, combined with Lemma \ref{lem:content-average} yields 
\begin{align*}
\cH^{d-1}_\infty(F_{n,\delta/2})&\leq  \cH^{d-1+\epsilon}_\infty(F_{n,\delta/2})^{\frac{d-1}{d-1+\epsilon}}  \\
&\lesssim_{\ell(Q)} \left( \int_{Q(n)} |\nabla u|^{p'} dx \right)^{\frac{d-1}{d-1+\epsilon}}.
\end{align*}

Let $F_{n,\delta}^*$ be the projection of  $F_{n,\delta}$ into $\mathbb R^{d-1}$. Then
\[
\cH^{d-1}_\infty(\bigcup_{n=M}^\infty F_{n,\delta}^*) \leq \sum_{n=M}^\infty \cH^{d-1}_\infty(F_{n,\delta/2}) \lesssim  \sum_{n=M}^\infty \left(\int_{Q(n\ell(\tilde{Q}))} |\nabla u|^{p'} dx\right)^{\frac{d-1}{d-1-\epsilon}} <\infty.
\]

Thus,
$\cH^{n-1}_\infty(\bigcap_{M=N}^\infty \bigcup_{n=M}^\infty F_{n,\delta}^*) = 0.$

\end{proof}

\subsection{Product weights} 

In the final part of the paper, we discuss the radial and product weight settings, where we can give necessary and sufficient conditions.

Let $1<p<d$.
For the following proof, we recall that for $v\in L^p(\mathbb R^d)$ the function \[\textbf M v(x) = \sup_{r>0} \avint_{B(x,r)}|v(x)|dx\] is the Hardy-Littlewood maximal function of $v$ at $x\in\mathbb R^d$. When $w\in \A_p(\mathbb R^d)$, we have that $\textbf M:L^p(\mathbb R^d,w)\to L^p(\mathbb R^d,w)$ is bounded, \cite[Chapter V]{Stein}. Further, we need a pointwise version of the Poincar\'e inequality: There exists a constant $C$ so that {for almost all} $x,y\in \R^d$ and for all $u \in W^{1,1}_{\loc}(\R^d)$ we have

\begin{equation}\label{PI-pp}
|u(x)-u(y)| \leq Cd(x,y)(\textbf M |\nabla u|(x) +  \textbf M|\nabla u|(y)).
\end{equation}

\begin{proof}[Proof of Theorem \ref{thm:pekka-mizuta-2nd}] First, we assume that $\sup_{r>0}\Rp(w_r)<\infty$, where $w_r(\overline x,t)=w(\overline x,t-r)$, and establish the existence of vertical limits. From Lemma \ref{lem:av-est}, since $\sup_{r>0}\Rp(w_r)<\infty$, we have that the unique almost sure radial limit $c$ exists, and for any $Q \subset \R^d$ and any $t\to\infty$ we have $\lim_{t\to\infty}\avint_{Q \times [t,t+\ell(Q)]} |u-c| d\mu = 0$.

 Fix $\epsilon>0$ and a cube $Q\subset \R^{d-1}$. Since $w\in \A_p(\R^d)$, $\textbf M |\nabla u|\in L^p(\mathbb R^d,w).$ 
 Clearly, for a.e. $\overline{x} \in Q$ we have $w_1(\overline{x})<\infty$. Thus, for almost every $\overline{x}\in Q$ there exists a  positive constant $T_{\overline{x}}\in(0,\infty)$ so that

\begin{equation}\label{eq:jointbound}\int_{t>T_{\overline{x}}} |\nabla u|^p + \left(\textbf M|\nabla u|(\overline{x},t)\right)^p w dt <\epsilon.
\end{equation}
Notice that $\textbf M |\nabla u| \geq |\nabla u|$ almost everywhere.  Since $w_2 \in \A_p(\R)$ we get using H\"older's inequality, for almost every $\overline{x}\in Q$ and each $t>T_{\overline{x}}$, that
\begin{equation}\label{eq:intbound}
\int_t^{t+1} |\nabla u|(\overline{x},s) ds \leq \int_t^{t+1} \textbf M |\nabla u|(\overline{x},s) ds\leq \frac{1}{w_1(\overline{x})w_2([t,t+1])}\epsilon^{\frac{1}{p}}.
\end{equation}

Let $Q_n = Q \times [n,n+1]$ for $n\in \N$. By the first paragraph, we have $\lim_{n\to\infty} \avint_{Q_n} |u-c|d\mu = 0$. By $\sup_{r>0} \Rp(w_r)<\infty$, and the definition of $\Rp(w)$, we have that $\inf_{r>0} w(B((\overline{0},r),2)) \gtrsim 1$. We also get $\inf_{n\in}w(Q_n)=w_1(Q)\inf_{n\in\N}w_2([n,n+1])>0$. Thus, $\inf_{n\in\N}w_2([n,n+1])>0$.

Thus, by doubling from Theorem \ref{thm:muckenhouppi}, we have that $\inf_{n\in\N} w(Q_n) >0$. Also:
\[
w(Q_n \cap \{|u-c|\geq \epsilon\}) \leq \frac{\avint_{Q_n} |u-c|d\mu }{\epsilon}w(Q_n).
\]

Now, for any $\delta \in (0,1/2)$, there is an $N_{\delta, Q}$ so that if $n\geq N_{\delta,Q}$ we have $w(Q_n \cap \{|u-c|\geq \epsilon\})\leq \delta w(Q_n)<w(Q_n)$. Thus, $w(Q_n \cap \{|u-c|< \epsilon\})\geq (1-\delta)w(Q_n) \geq \frac{1}{2}w(Q_n)$. Since $\inf_{n\in\N} w(Q_n) >0$, and since $Q_n$ have disjoint interiors, we get $\lim_{n\to\infty} \avint_{Q_n} \textbf M |\nabla u|^p d\mu = 0$. In particular, there is an index $N$ so that for $n\geq N$, we have $\avint_{Q_n} \textbf M |\nabla u|^p d\mu \leq \epsilon$. Thus, by the Markov inequality, for $n\geq \max(N,N_{\delta,Q})$, there must exist a point $y_n\in Q_n \cap \{|u-c|< \epsilon\}$ so that 
\begin{equation}\label{eq:ynbound}
\mathbf M |\nabla u|(y_n) \leq 2\epsilon^{1/p} \text{ and } |u(y_n)-c|\leq \epsilon.
\end{equation}

By equation \eqref{eq:intbound} for almost every $\overline{x}$, if $n\geq \max\{T_{\overline{x}}\}$, then there is a value $t_{n,\overline{x}}\in [n,n+1]$ with 

\begin{equation}\label{eq:tnbound}
\mathbf M |\nabla u|(\overline{x},t_{n,\overline{x}}) \leq \frac{\epsilon^{\frac{1}{p}} }{w_1(\overline{x})w_2([n,n+1])}.
\end{equation}

Consider such a $\overline{x}$ and let $t> \max\{T_{\overline{x}},N_{\delta,Q}, N\}$. Choose $n \geq \max\{N,N_{\delta,Q}\}$ so that $n \leq t <n+1$. Combining the bounds \eqref{eq:ynbound}, \eqref{eq:tnbound} and \eqref{PI-pp}, we have

\begin{equation}\label{eq:yntnbound}
|u(y_n)-u(\overline{x},t_{n,\overline{x}})|\lesssim \diam(Q_n) (\mathbf M|\nabla u|(y_n) + \mathbf M|\nabla u|(\overline{x},t_{n,\overline{x}})) \lesssim (\diam(Q)+1 ~ ) (\epsilon^{\frac{1}{p}} / ( w_1(\overline{x})w_2(Q)) + 2\epsilon)).
\end{equation}

By Lemma \ref{lem:acclines} the function $t\mapsto u(\overline{x}, t)$ is absolutely continuous for almost every $\overline{x}$ and 

\begin{equation}\label{eq:tbound}
|u(\overline{x},t_{n,\overline{x}})-u(\overline{x},t)|\leq \int_n^{n+1} |\nabla u|(\overline{x},s) ds \leq  \frac{1}{w_1(\overline{x})w_2([n,n+1])}\epsilon^{\frac{1}{p}}.
\end{equation}

By combining estimates \eqref{eq:ynbound}, \eqref{eq:tbound} and \eqref{eq:yntnbound} with the triangle inequality, we obtain
\[
|u(\overline{x},t)-c| \lesssim \frac{\diam(Q) + 2}{w_1(\overline{x})w_2([n,n+1])}\epsilon^{\frac{1}{p}} + (\diam(Q) + 2 ~ )\epsilon.
\]

Since $\inf_{n\in\N}w_2([n,n+1])>0$ and since $\epsilon>0$ was arbitrary the existence of vertical limits almost everywhere follows.

Finally, the proof for the converse direction follows from Lemma \ref{Syl}. By this result, if $\sup_{r>0}\Rp(w_r)=\infty$, there exists a $u \in \Wdp(\R^d,w)$ which does not have any vertical limits in a set of positive measure.

\end{proof}

\subsection{Radial weights}

\begin{proof}[Proof of Theorem \ref{thm:radial-weights}]

First, the implication $(2) \Longrightarrow (1)$ is shown by the following argument which uses contrapositive and  Example \ref{example:rough}. Indeed, if $\inf_{r>0} \int_r^{r+1} v(s) ds = 0$, then we can find a sequence of $r_i$, so that $\lim_{j\to\infty} \int_{r_j}^{r_j+1} v(s) ds = 0$ and  $r_j \to \infty$. Consider the cube $Q=Q(0,1)\subset \R^{d-1}$. Then, by doubling, we can show that $\lim_{j\to \infty} w(Q\times [r_j,r_j+1])=0$. Now, example \ref{example:rough} furnishes a function $u$ without vertical limits for any $\overline{x}\in Q$.

Next, we turn to establish $(1) \Longrightarrow (2)$. This proof is nearly the same as that of Theorem \ref{thm:pekka-mizuta-2nd}. Similarly to that argument, fix an $\epsilon>0$ and a cube $Q \subset \R^{d-1}$. We indicate the few differences from this proof.

First, the assumption $\inf_{r>0} \int_r^{r+1} v(s) ds > 0$ implies $\inf_{\ell(\mathbf{Q})=1}w(\mathbf{Q})>1$ and so $\sup_{t>0}\Rp(w_t)<\infty$.  With this addition, the first paragraph of the proof in Theorem \ref{thm:pekka-mizuta-2nd} applies, and  $\lim_{t\to\infty} \avint_{Q \times [t,t+\ell(Q)]} |u-c|d\mu = 0$.

By the boundedness of the maximal operator for Muckenhoupt weights, we have $\mathbf M|\nabla u| \in L^p(\mathbb R^d,w)$. Therefore, for almost every $\overline{x}\in Q$ there exists a $T_{\overline{x}}>0$ so that
\begin{equation}\label{eq:jointbound-radial}
\int_{t>T_{\overline{x}}} |\nabla u|^p + \left(\mathbf M|\nabla u|(\overline{x},t)\right)^p w(\overline{x},t) dt <\epsilon.
\end{equation}
This bound replaces \eqref{eq:jointbound} in the proof of Theorem \ref{thm:pekka-mizuta-2nd}. 

The function $v_2(t)=v(t^{1/d})$ satisfies $v_2\in \A_p(\R)$ by \cite{DMOSradial} and $w(\overline{x},t) = v_2(\sqrt{|\overline{x}|^2+t^2}^d).$  Let $t>\max\{T_{\overline{x}},|\overline{x}|\}$, and let $s_1=\sqrt{|\overline{x}|^2+t^2}^{d}$ and $s_2=\sqrt{|\overline{x}|^2+(t+1)^2}^{d}$. Since $v_2 \in \A_p(\R)$, there is a constant $C$ so that

\begin{equation}\label{eq:apv2}
\frac{1}{s_2-s_1}\int_{s_1}^{s_2} v_2(s) ds \left(\frac{1}{s_2-s_1}\int_{s_1}^{s_2} v_2^{\frac{1}{1-p}}(s) ds\right)^{\frac{1}{p-1}} \leq C.
\end{equation}

By a change of variables, $t\mapsto \sqrt{|\overline{x}|^2+t^2}^{d}$, we get $\int_{s_1}^{s_2}v_2(s) ds \gtrsim t^{d-1}\int_{t}^{t+1} w(\overline{x},s) ds$ and $\int_{s_1}^{s_2} v_2^{\frac{1}{1-p}}(s)ds \gtrsim  t^{d-1}\int_{t}^{t+1} w(\overline{x},s)^{\frac{1}{1-p}} ds$. We also have $s_2-s_1 \lesssim t^{d-1}$. Thus, by changing the constant $C$, we get
\begin{equation}\label{eq:apw}
\int_{t}^{t+1} w(\overline{x},s) ds \left(\int_{t}^{t+1} w(\overline{x},s)^{\frac{1}{1-p}} ds\right)^{\frac{1}{p-1}} \leq C.
\end{equation}
By another change of variables and the assumption, we get a constant $\delta=\delta({\overline{x}})>0$ so that $\int_{t}^{t+1} w(\overline{x},t) dt \geq \delta$ for all $t>1$.  Combining this with \eqref{eq:apw}  gives $\left(\int_{t}^{t+1} w(\overline{x},s)^{\frac{1}{1-p}} ds\right)^{\frac{1}{p-1}} \leq C_{\overline{x}}$ for some constant $C_{\overline{x}}$ and all $t>\max\{T_{\bar x}, |\bar x|,1\}$. This, together with H\"older's inequality and estimate \eqref{eq:jointbound-radial} yields for all $t\geq \max\{1,|\overline{x}|,T_{\overline{x}}\}$ that
\begin{equation}\label{eq:intbound-radial}
\int_t^{t+1} |\nabla u|(\overline{x},s) ds \leq \int_t^{t+1} \mathbf M |\nabla u|(\overline{x},s) ds\leq C_{\overline{x}}\epsilon^{\frac{1}{p}}.
\end{equation}

With $C_{\overline{x}}$ replacing $\frac{1}{w_1(\overline{x})w_2(Q)}$, and with the additional restriction that $t>\max\{1,|\overline{x}|\}$ the rest of the proof of Theorem \ref{thm:pekka-mizuta-2nd} applies without further changes.
\end{proof}

\bibliographystyle{alpha}
\bibliography{pmodulus}

\end{document}